\documentclass[smallextended,referee,envcountsect,]{svjour3}
\smartqed
\usepackage{graphicx}
\usepackage{amssymb,amsmath}
\usepackage{float}
\usepackage{color,fullpage}
\usepackage{cite}
\usepackage[title,titletoc]{appendix}
\usepackage{algorithm,algorithmic}
\UseRawInputEncoding
\usepackage{graphics,graphicx,epstopdf}

\usepackage{hyperref}
\hypersetup{
	colorlinks=true,
	linkcolor=blue,
	filecolor=magenta,      
	urlcolor=cyan,
	citecolor = blue,
}

\definecolor{blue}{rgb}{0,0,0.9}
\definecolor{red}{rgb}{0.9,0,0}
\definecolor{green}{rgb}{0,0.9,0}

\newcommand{\be}{\begin{equation}}
\newcommand{\ee}{\end{equation}}
\newcommand{\ben}{\begin{enumerate}}
\newcommand{\een}{\end{enumerate}}
\newcommand{\bfr}{\begin{frame}}
\newcommand{\efr}{\end{frame}}
\newcommand{\bit}{\begin{itemize}}
\newcommand{\eit}{\end{itemize}}
\newcommand{\argmin}{\mathop{\rm argmin}}

\newcommand{\R}{\mathbb R}
\newcommand{\E}{\mathbb{E}}

\newcommand{\cB}{\mathcal B}
\newcommand{\cC}{\mathcal C}
\newcommand{\cD}{\mathcal D}

\newcommand{\cF}{\mathcal F}

\newcommand{\cH}{\mathcal H}

\newcommand{\cL}{\mathcal L}

\newcommand{\cS}{\mathcal S}

\newcommand{\cX}{\mathcal X}
\newcommand{\cY}{\mathcal Y}

\def\norm#1{\left\|#1\right\|}
\def\inner#1{\left\langle#1\right\rangle}
\def\inprod#1#2{\langle#1,\,#2\rangle}

\newtheorem{assumption}{Assumption}[section]

\def\dist{{\rm dist}}
\def\dom{{\rm dom}}

\journalname{}

\begin{document}

\title{Stochastic Bregman Subgradient Methods for Nonsmooth Nonconvex Optimization Problems}


\author{Kuangyu Ding \and Kim-Chuan Toh}

\institute{Kuangyu Ding \at
             National University of Singapore \\
              10 Lower Kent Ridge Rd, 119076, Singapore\\
              kuangyud@u.nus.edu
           \and
              Kim-Chuan Toh, \at
             National University of Singapore \\
              10 Lower Kent Ridge Rd, 119076, Singapore\\
              mattohkc@nus.edu.sg
}

\date{Received: date / Accepted: date}

\maketitle

\begin{abstract}
This paper focuses on the problem of minimizing a locally Lipschitz continuous function. Motivated by the effectiveness of Bregman gradient methods in training nonsmooth deep neural networks and the recent progress in stochastic subgradient methods for nonsmooth nonconvex optimization problems \cite{bolte2021conservative,bolte2022subgradient,xiao2023adam}, we investigate the long-term behavior of stochastic Bregman subgradient methods in such context, especially when the objective function lacks Clarke regularity. We begin by exploring a general framework for Bregman-type methods, establishing their convergence by a differential inclusion approach. For practical applications, we develop a stochastic Bregman subgradient method that allows the subproblems to be solved inexactly. Furthermore, we demonstrate how a single timescale momentum can be integrated into the Bregman subgradient method with slight modifications to the momentum update. Additionally, we introduce a Bregman proximal subgradient method for solving composite optimization problems possibly with constraints, whose convergence can be guaranteed based on the general framework. Numerical experiments on training nonsmooth neural networks are conducted to validate the effectiveness of our proposed methods. 
\end{abstract}

\vspace{1em}
\noindent{Communicated by Sorin-Mihai Grad.}


\section{Introduction}

In this paper, we focus on exploring the application of the Bregman distance based stochastic subgradient methods for solving nonsmooth nonconvex optimization problems. We consider the following unconstrained optimization problem,
\begin{equation}\label{min-prob}
\min_{x\in\R^n} f(x),
\end{equation}
where $f:\R^n\rightarrow\R$ is a locally Lipschitz continuous function (possibly nonconvex and nonsmooth). This class of problems encompasses many important applications, particularly in training nonsmooth neural networks where nonsmooth activate functions, such as the rectified linear unit (ReLU), are employed. 

First order methods are commonly used to solve \eqref{min-prob}. In many scenarios, only noisy gradients or subgradients are available. During the last several decades, a great number of stochastic first order methods have been proposed. Among these, stochastic gradient descent (SGD) could be the most fundamental method. Based on SGD, a great number of its variants have also been developed to obtain some benefits on speed-up, stability, or memory efficiency. For example, the heavy-ball SGD \cite{polyak1964some}, signSGD \cite{bernstein2018signsgd}, and normalized SGD \cite{you2017large,you2019large,cutkosky2020momentum}. These methods have demonstrated great efficiency and competitive generalization performance in various tasks in deep learning. 

Despite the progress in developing SGD-type methods, the conventional convergence analysis largely pertains to scenarios where the objective function $f$ exhibits certain regularity properties, namely, differentiability or weakly convexity. However, the prevalence of nonsmooth activation functions, such as ReLU or Leaky ReLU, in neural network architectures results in loss functions that often lack Clarke regularity (e.g. differentiability, weak convexity).  Consequently, the conventional convergence analysis of SGD-type methods is not applicable in the context of training nonsmooth neural networks. {Towards this issue, Bolte and Pauwels \cite{bolte2021conservative} introduce the concept of conservative fields, which generalizes the Clarke subdifferential and admits a chain rule and sum rule for path-differentiable functions that may not be Clarke regular. Leveraging the concept of conservative fields and path-differentiability, subsequent studies \cite{bolte2021conservative,castera2021inertial,bolte2022subgradient,le2023nonsmooth,xiao2023adam,xiao2023convergence,ding2023adam,xiao2024developing,gurbuzbalaban2022stochastic,ruszczynski2021stochastic} utilize a differential inclusion approach \cite{benaim2005stochastic,borkar2009stochastic} to establish convergence for SGD-type methods in training nonsmooth deep neural networks. In particular, \cite{bolte2021conservative,bolte2022subgradient} study the convergence properties of SGD and proximal SGD for minimizing nonsmooth, path-differentiable functions. Moreover, \cite{castera2021inertial} proposes the inertial Newton algorithm (INNA), which is regarded as a variant of momentum-accelerated SGD methods. Additionally, \cite{le2023nonsmooth} establishes the convergence of SGD with heavy-ball momentum, and \cite{xiao2023convergence} establishes the global stability of SGD with heavy-ball momentum. In addition, \cite{gurbuzbalaban2022stochastic,ruszczynski2021stochastic} design stochastic subgradient methods for solving multi-level nested optimization problems, and \cite{xiao2023adam,ding2023adam} establish the convergence properties of adaptive subgradient methods widely used in deep learning. Moreover, \cite{xiao2024developing} develops subgradient methods to solve nonsmooth nonconvex constrained optimization problems. In summary, the crucial step in establishing convergence is to construct the differential inclusion corresponding to the iterative optimization algorithm and then to find a suitable Lyapunov function associated with that differential inclusion.}

Beyond SGD-type methods, stochastic Bregman gradient methods, also known as mirror descent methods, have gained increasing interest recently. Initially introduced by Nemirovski and Yudin \cite{nemirovskij1983problem} for solving constrained convex problems, the fundamental update scheme of Bregman gradient method is given as follows,
\begin{equation}
    x_{k+1}=\argmin_{x\in\R^n}\left\{\inner{g_k,x-x_k}+\frac{1}{\eta_k}\cD_\phi(x,x_k)\right\},
\end{equation}
where $\cD_\phi$ represents the Bregman distance, a generalized measure of distance provided by a kernel function $\phi$, with its definition to be presented in Section 2. A distinctive feature of Bregman gradient methods is the utilization of a broader selection of kernel functions, extending beyond the classical kernel function of $\frac{1}{2}\norm{\cdot}^2$ used in SGD. Therefore, SGD is a special instance of Bregman gradient methods. The application of Bregman gradient methods has expanded to encompass both convex and nonconvex problems, as well as deterministic and stochastic contexts, as demonstrated by subsequent works \cite{beck2003mirror,bauschke2017descent,lu2018relatively,bolte2018first,zhang2018convergence,dragomir2021optimal,hanzely2021fastest,yang2022bregman,ding2023nonconvex}. Bregman gradient methods have found many applications in optimization problems associated with probability constraints, such as optimal transport problem \cite{benamou2015iterative,chu2023efficient} and reinforcement learning \cite{lan2023policy,zhan2023policy}. Recent works have increasingly highlighted the potential of Bregman gradient methods in neural network training. Notably, \cite{ding2023nonconvex} has shown that Bregman gradient methods can achieve better generalization performance and enhanced robustness in stepsize tuning compared to SGD for certain deep learning tasks. Additionally, many recent works, such as \cite{gunasekar2018characterizing, wu2020continuous, li2021implicit, azizan2021stochastic, sun2023unified}, have demonstrated that Bregman gradient methods exhibit implicit regularization, leading to improved generalization performance with the selection of an appropriate kernel function. Additionally, \cite{amid2020reparameterizing,ghai2022non,li2022implicit} provide an interpretation of the classical gradient descent method on reparameterized models from the perspective of Bregman gradient methods. This perspective has led to some applications in deep learning, including neural network quantization, as illustrated by \cite{ajanthan2021mirror}. 

However, existing convergence analysis for Bregman gradient methods has been limited to objective functions that are either differentiable or weakly convex. Given the increasing interest in applying Bregman gradient methods to train nonsmooth neural networks, and the current limited understanding of the convergence properties of Bregman subgradient methods for solving \eqref{min-prob}, this paper aims to provide a theoretical convergence guarantee for Bregman subgradient methods applied to nonsmooth nonconvex problems, particularly for the training of nonsmooth neural networks. Moreover, we propose practical Bregman subgradient methods specifically for training nonsmooth neural networks while ensuring their convergence.

To summarize, the contributions of this paper are as follows:
\begin{itemize}
    \item \textbf{General Bregman differential inclusion:} We first investigate a general Bregman differential inclusion, whose discrete update scheme is consistent with that of Bregman-type methods and allows for biased evaluations of the abstract set-valued mapping. We establish the convergence for the discrete update scheme associated with this Bregman differential inclusion. Specifically, we demonstrate that any cluster point of the sequence generated by the discrete update scheme lies in the stable set of the Bregman differential inclusion, and the Lyapunov function values converge. A key aspect of our approach is the utilization of linear interpolation of the dual sequence induced by the kernel function.
    \item \textbf{Applications of the general Bregman differential inclusion:} By exploiting the flexibility of choosing the kernel function and the set valued mapping in the general framework, we introduce three types of stochastic Bregman subgradient methods for different scenarios. First, we consider the vanilla stochastic Bregman subgradient method (SBG) for unconstrained optimization problem, and establish its convergence. We further show that under certain regularity conditions of the kernel function, a preconditioned subgradient method can be regarded as an inexact Bregman subgradient method, thus fitting it within our proposed method. Moreover, we propose a single timescale  momentum based stochastic Bregman subgradient method (MSBG), and establish its convergence by the sophisticated choice of the kernel function and set-valued mapping. Lastly, we consider the stochastic Bregman proximal subgradient method (SBPG) for constrained composite optimization problem and establish its convergence.
    \item \textbf{Numerical experiments:} To evaluate the performance of our proposed stochastic Bregman subgradient methods, we employ a block-wise polynomial kernel function in our proposed Bregman subgradient methods. We conduct numerical experiments to compare SGD, the momentum based stochastic Bregman subgradient method (MSBG), and its inexact version (iMSBG), for training nonsmooth neural networks. The results illustrate that our Bregman subgradient methods can achieve generalization performance comparable to that of SGD as well as the enhanced robustness of stepsize tuning.
\end{itemize}
 
The remaining sections of the paper are structured as follows. In Section 2, we provide preliminary materials on notations, nonsmooth analysis, and Bregman proximal mapping. Section 3 is dedicated to establishing the convergence of the general Bregman-type method based on the associated differential inclusion. The specific applications of the general Bregman-type method are presented in Section 4, where we first consider the vanilla stochastic Bregman subgradient method and then propose a single timescale momentum based Bregman subgradient method. The latter part of Section 4 explores the stochastic Bregman proximal subgradient method. Section 5 conducts numerical experiments to demonstrate the performance of our proposed Bregman subgradient methods. In the last section, we give a conclusion on this paper.

\section{Preliminary}\label{preliminary-section}
\subsection{Notations}\label{sec:notation}
Given a proper and lower semicontinuous function $f:\mathbb{R}^n\rightarrow\overline{\mathbb{R}} := (-\infty,\infty]$, we denote its domain as $\dom\,f=\{x:f(x)<\infty\}$. The Fenchel conjugate function of $f$ is defined as $f^*(y):=\sup\{\langle x,y\rangle - f(x):x\in\mathbb{R}^n\}$. For a set $\mathcal{S}\subset\mathbb{R}^n$, we use ${\rm cl\,\mathcal{S}}$ to denote its closure, ${\rm int}\,\mathcal{S}$ to denote the set of its interior points, and ${\rm conv}(\cS)$ to denote its convex hull. A function $f:\cS\rightarrow\R$ is said to be of class ${\cC}^k(\mathcal{S})$ if it is $k$ times differentiable and the $k$-th derivative is continuous on $\mathcal{S}$. When there is no ambiguity regarding the domain, we simply use the notation ${\cC}^k$. We let $\mathcal{C}(A,B)$ be the set of continuous mappings from set $A$ to set $B$. We use $\|\cdot\|$ to denote the Euclidean norm for vectors and the Frobenius matrix norm for matrices. The $d$-dimensional unit ball is denoted by $\mathbb{B}^d$. The distance between a point $w$ and a set $A$ is denoted by $\dist(w,A):=\inf\{\|w-u\|:\;u\in A\}$. We use the convention $\dist(w,\emptyset):=\infty$. {The Minkowski sum of two sets $A$ and $B$ is denoted as 
$A+B:=\{u+v:\;u\in A,\;v\in B\}$. We use $\alpha A:=\{\alpha u:\;u\in A\}$ to denote the scaled set of $A$ by a given scalar $\alpha$.} For a positive sequence $\{\eta_k\}$, we define $\lambda_\eta(0):=0$, $\lambda_\eta(k):=\sum_{i=0}^{k-1}\eta_i$ for $k\geq1$, and $\Lambda_\eta(t):=\sup\{k:\lambda_\eta(k)\leq t\}$. In other word, $\Lambda_\eta(t)=k$ if and only if $\lambda_\eta(k)\leq t<\lambda_\eta(k+1)$. {Given a convex function $g$, we say $x$ is a $\nu$-optimal solution of the problem $\min_{x\in\R^n}\;g(x)$ if ${\rm dist}(0,\partial g(x))\leq\nu$.}

Let $(\Omega,\mathcal{F},\mathbb{P})$ be a probability space. Consider a stochastic process $\{\xi_k\}_{k\geq0}$ and a filtration $\{\mathcal{F}_k\}_{k\geq0}$, where $\mathcal{F}_k$ is defined by the $\sigma$-algebra $\mathcal{F}_k:=\sigma(\xi_0,\ldots,\xi_k)$ on $\Omega$, the conditional expectation is denoted as $\mathbb{E}[\cdot|\mathcal{F}_k]$.

    \subsection{Nonsmooth analysis}
In this subsection, we present some concepts from nonsmooth analysis, mainly based on \cite{clarke1990optimization,RockWets98}.
\begin{definition}\label{graph-def}
Let $\mathcal{X}\subset\mathbb{R}^n$. A set-valued mapping $S:\mathcal{X}\rightrightarrows\mathbb{R}^m$ is said to be closed if its graph, defined by
\[
{\rm graph}(S):=\{(x,y)\in\R^n\times\R^m:y\in S(x)\}
\]
is a closed set in $\R^n\times\R^m$. Give a nonnegative $\delta$, the $\delta$-perturbed set of $S$ is defined by 
\[
S^\delta(x):=\cup_{\{y\in\R^n:\norm{y-x}\leq\delta\}}\left(S(y)+\delta\mathbb{B}^m\right).
\]
$S:\mathcal{X}\rightrightarrows\mathbb{R}^m$ is called outer semicontinuous at $\bar{x}\in\mathcal{X}$ if for any sequence $x_i\overset{\mathcal{X}}{\rightarrow}\bar x$ and $v_i\in S(x_i)$ converging to some $\bar{v}\in\mathbb{R}^m$, we have $\bar{v}\in S(\bar{x})$. $S$ is said to be outer semicontinuous if it is outer semicontinuous everywhere over $\mathcal{X}$.
\end{definition}


\begin{definition}
Let $\cX\subset\R^n$ be a closed set, the regular normal cone at $\bar x\in\cX$ is defined as $\hat N_{\cX}(\bar x):=\{v\in\R^n:\inner{v,x-\bar x}\leq o(\norm{x-\bar x}),\text{ for $x\in\cX$}\}$. The limiting normal cone is defined by $N^L_{\cX}(\bar x):=\{v\in\R^n:x_k\overset{\cX}{\rightarrow}\bar x,v_k\in\hat N_{\cX}(x_k),v_k\rightarrow v\}$. The normal cone is defined by $N_{\cX}(\bar x):={\rm cl\,conv}(N^L_{\cX}(\bar x))$.
\end{definition}
According to \cite[Proposition 6.6]{RockWets98}, the limiting normal cone $N^L_\cX$ is outer semicontinuous, while this is not necessarily true for $N_\cX$.

\begin{definition}\label{Clarke-def}
Let $f:\R^n\rightarrow\R$ be a locally Lipschitz continuous function. The generalized directional derivative of $f$ at $x\in\R^n$ along the direction $d\in\R^n$ is defined by
\[
f^\circ(x;d):=\lim\sup_{y\rightarrow x,t\downarrow0}\frac{f(y+td)-f(y)}{t}.
\]
The Clarke subdifferential of $f$ at $x$ is defined by
\[
\partial f(x):=\left\{v\in\R^n:\inprod{v}{d}\leq f^\circ(x;d),\;for\;all\;d\in\R^n\right\}.
\]
$f$ is Clarke regular, if for any $d\in\R^n$, its one-side directional derivative, defined by
\[
f^*(x;d):=\lim_{t\downarrow0}\frac{f(x+td)-f(x)}{t},
\]
exists and $f^\circ(x;d)=f^*(x;d)$.
\end{definition}

Clarke regularity excludes functions whose graph has upwards corners, such as $-|x|$. Notably, some basic calculus rules, including the sum and chain rules, may fail for the Clarke subdifferential without Clarke regularity. This limitation motivates the introduction of conservative fields.

\subsubsection{Conservative field, path differentiability}
In this part, we briefly introduce some relevant materials on conservative field, which are mainly based on \cite{bolte2021conservative}.
\begin{definition} (Conservative field and path-differentiability)\label{conservative-field-def}
Let $f:\R^n\to\R$ be a locally Lipschitz function. 
A set-valued mapping $D_f:\R^n\rightrightarrows\R^n$ is a conservative field for $f$ if it is nonempty closed valued, and has closed graph. For any absolutely continuous curve $\gamma:[0,1]\rightarrow\R^n$, $f$ admits a chain rule with respect to $D_f$, i.e.
\be
\frac{d}{dt}(f\circ\gamma)(t)=\inprod{v}{\dot\gamma(t)},\text{ for all }v\in D_f(\gamma(t)) \text{ and almost all }t\in[0,1].
\label{conservative-def}
\ee
If a locally Lipschitz function $f$ admits a conservative field $D_f$, 
{then} we say that $f$ is path-differentiable, and $f$ is the potential function of $D_f$.
\end{definition}

{
More generally, we can define conservative mappings for vector-valued functions, which serve as a generalization of the Jacobian.
\begin{definition}\label{conservative-mapping-def}
Let $F:\R^n\rightarrow\R^m$ be a locally Lipschitz continuous function. $J_F:\R^n\rightrightarrows\R^{m\times n}$ is called a conservative mapping for $F$, if for any absolutely continuous curve $\gamma:[0,1]\rightarrow\R^n$, the following chain rule holds {for} almost all $t\in[0,1]$:
\[
\frac{d}{dt}F(\gamma(t))=V\dot\gamma(t),\;for\;any\;V\in J_F(\gamma(t)).
\]
\end{definition}
}

\begin{definition}
\label{def:chain_rule_set}
 Given a closed set $\cX\subset\R^n$, we say that $\cX$ admits a chain rule, if for any absolutely continuous curve $\gamma:[0,1]\rightarrow\cX$, it holds that 
 \[
 \inner{N_\cX(\gamma(t)),\dot\gamma(t)}=\{0\}, \text{ for almost all $t\in[0,1]$}.
 \]
\end{definition}
As demonstrated in \cite{davis2020stochastic}, when $\cX$ is Whitney stratifiable (e.g. $\cX$ is definable), then $\cX$ admits the chain rule. {We also use $\frac{d}{dt}(f\circ\gamma)(t)=\inprod{D_f(\gamma(t))}{\dot\gamma(t)} \text{ for almost all }t\in[0,1]$ to represent that the chain rule \eqref{conservative-def} is valid for almost all $t\in[0,1]$.} As stated in \cite[Remark 3]{bolte2021conservative}, if $D_f$ is a conservative field for $f$, then ${\rm conv}(D_f)$ is also a conservative field for $f$, and $D_f$ is locally bounded. Consequently, it follows from \cite[Theorem 5.7, Proposition 5.12]{RockWets98} that for any $x\in\R^n$ and $\varepsilon>0$, there exist a neighborhood $V$ of $x$, such that $\cup_{y\in V}D_f(y)\subset D_f(x)+\epsilon\mathbb{B}^n$. 
It is important to note that the conservative field for a function \(f\) is not unique. Lemma \ref{le:CF_Clark} clarifies the relationship between a conservative field and the Clarke subdifferential, showing that the Clarke subdifferential is the smallest conservative field among all convex-valued conservative fields. {We define \(x\) to be a \(D_f\)-stationary point if \(0 \in D_f(x)\) to distinguish it from the conventional stationary point defined via the Clarke subdifferential. According to Lemma \ref{le:CF_Clark}, a \(D_f\)-stationary point represents a weaker form of stationarity compared to a Clarke-stationary point. This weaker notion is inevitable if we want to capture the "subgradients" produced by automatic differentiation (AD) in deep learning, for which the Clarke subdifferential may not capture in the absence of Clarke regularity. The same stationarity measure is also used in \cite{bolte2022subgradient,le2023nonsmooth,xiao2023adam}. 

Additionally, \cite{bolte2021conservative} highlights the importance of the conservative field generated by the AD algorithm, widely utilized in modern deep learning frameworks like TensorFlow, PyTorch, and JAX. In the context when the conservative field corresponds precisely to the AD-generated subdifferential, we refer to such a $D_f$-stationary point as a stationary point in the sense of AD, which is the typical case in deep learning. To illustrate the conservative field generated by the AD algorithm, consider the loss function of a two-layer ReLU neural network:
\[
f(W):=\frac{1}{2}\norm{W_2\sigma(W_1x)}^2,
\]
where $\sigma(a):=\max\{a,0\}$ denotes the (coordinate-wise) ReLU activation function,  $x$ is fixed, and $W=(W_1,W_2)$. In Pytorch, the subgradient computed by AD corresponds to $\sigma'(a)=\left\{\begin{array}{cc}
    1, & a>0, \\
    0, & a\leq0.
\end{array}\right.$ Consequently, if $D_f$ represents the AD-generated conservative field, we have
\[
D_f(W_1,W_2)=\left(\left(\sigma'(W_1x)\circ (W_2^TW_2\sigma(W_1x))\right)x^T,\;W_2\sigma(W_1x)\left(\sigma(W_1x)\right)^T\right).
\]
\begin{lemma}(\cite[Theorem 1, Corollary 1]{bolte2021conservative})
\label{le:CF_Clark}
    Consider a locally Lipschitz continuous function $f:\R^n\rightarrow\R$, with $\partial f$ as its Clarke subdifferential and $D_f$ as its convex-valued conservative field. Then, 
    \begin{enumerate}
        \item $D_f(x)=\{\nabla f(x)\}$ for almost all $x\in\R^n$.
        \item $\partial f$ is a conservative field of $f$. Moreover, for any $x\in\R^n$, it holds that $\partial f(x)\subset D_f(x)$.
    \end{enumerate} 
 \end{lemma}
}

The next lemma provides the motivation for introducing the concept of conservative field, as it highlights that conservative fields preserve some basic calculus rules that do not necessarily hold for Clarke subdifferential without Clarke regularity. 

\begin{lemma}(\cite[Lemma 5, Corollary 4]{bolte2021conservative})\label{le:chain-rule-conservative}
\begin{enumerate}
    \item Let $F_1:\R^n\rightarrow\R^m$, $F_2:\R^m\rightarrow\R^l$ be locally Lipschitz continuous vector-valued functions, and $J_1:\R^n\rightrightarrows\R^{m\times n}$ and $J_2:\R^m\rightrightarrows\R^{l\times m}$ be the conservative mappings of $F_1$ and $F_2$ respectively. Then the mapping $J_2\circ J_1:\R^n\rightrightarrows\R^{l\times n}$ is a conservative mapping for $F_2\circ F_1$. 
    
    \item Let $D_{f_i}$ be a conservative field for $f_i$, $i=1,...,N$. Then, $D_f=\sum_{i=1}^ND_{f_i}$ is a conservative field of $f=\sum_{i=1}^Nf_i$.

\end{enumerate}
\end{lemma}

The Clarke subdifferential may fail to satisfy key calculus rules, whereas the conservative field retains these properties, as established in Lemma \ref{le:chain-rule-conservative}. This ensures the well-definedness of gradient methods based on automatic differentiation (AD) in deep learning，since "subdifferential" generated by AD algorithms is a special case of conservative field. A notable subclass of path-differentiable functions is the set of definable functions in o-minimal structures \cite{coste2000introduction,van1996geometric}, which includes widely used activation and loss functions, such as sigmoid, ReLU, softmax, $l_1$ loss, cross-entropy, and hinge loss. Definable functions are closed under composition, summation, and set-valued integration, ensuring that neural networks built from such blocks remain definable \cite{davis2020stochastic,bolte2021conservative}. Their geometric structure, characterized by Whitney stratification, enables effective analysis of nonsmooth optimization algorithms \cite{bolte2007clarke,davis2020stochastic,bolte2021conservative}.

\subsection{Bregman proximal mapping}
In this subsection, we review some concepts related to Bregman proximal 
mappings. Given that this paper mainly focuses on unconstrained problems, we restrict our discussion on the entire space $\R^n$. For more general concepts about Bregman gradient mapping, readers can refer to works such as \cite{bauschke2017descent,bolte2018first,ding2023nonconvex}.

\begin{definition}\label{Bregman-distance-def}
(Kernel function and Bregman distance over $\R^n$).  A function $\phi:\R^n\rightarrow\R$ is called a kernel function over $\R^n$, if $\phi$ is convex and $\phi\in\cC^1(\R^n)$. The Bregman distance \cite{bregman1967relaxation} generated by $\phi$ is denoted as $\cD_\phi(x,y):\R^n\times\R^n\rightarrow[0,+\infty)$, where
\[
\cD_\phi(x,y)=\phi(x)-\phi(y)-\inprod{\nabla\phi(y)}{x-y}.
\]
\end{definition}

\begin{definition}\label{Legendre}
(Legendre kernel over $\R^n$). Let $\phi$ be a kernel function over $\R^n$, such that $\lim_{k\rightarrow\infty}\|\nabla\phi({x_k})\|=\infty$ whenever $\{{x_k}\}_{k\in\mathbb N}$ satisfies $\lim_{k\rightarrow\infty}\norm{x_k}=\infty$. The function $\phi$ is called a Legendre function over $\R^n$ if it is also {strictly} convex on $\R^n$.
\end{definition}

\begin{definition}\label{Bregman-prox-maping}
Given a locally Lipschitz continuous function $R$ and a Legendre kernel function $\phi\in\cC^1(\R^n)$,  we denote the Bregman proximal mapping by ${\rm Prox}^\phi_R:=(\nabla\phi+\partial R)^{-1}\nabla\phi$, which is a set-valued mapping defined as follows,
\begin{equation}
{\rm Prox}^\phi_R(x):=\argmin_{u\in\R^n}\;\{R(u)+\cD_\phi(u,x)\}.
\label{Breg-prox-def}
\end{equation}  
\end{definition}
Under mild conditions, ${\rm Prox}^\phi_R(x)$ is a nonempty compact set for any $x\in\R^n$, {which will ensure} the well-posedness of our methods. We have the following lemma, which directly follows from Weierstrass's theorem, we omit its proof for simplicity. 

\begin{lemma}\label{well-def-prox-mapping}
 Let $R$ be a continuous function and $\phi$ be a Legendre function over $\R^n$. Suppose $\alpha R+\phi$ is supercoercive, i.e. $\lim_{\|x\|\rightarrow\infty}\frac{\alpha R(x)+\phi(x)}{\|x\|}=\infty$, for any $\alpha>0$. Then, for any $x\in\R^n$, the set ${\rm Prox}^\phi_R(x)$ is a nonempty compact subset of $\R^n$.
\end{lemma}
In the context of Bregman proximal gradient method, we focus on the Bregman forward-backward splitting operator ${\rm T}^{\phi}_{\alpha,R}:\R^n\times\R^n\rightrightarrows\R^n$ defined by:
\[
{\rm T}^{\phi}_{\alpha,R}(x,v):=\argmin_{z\in\R^n}\;\left\{\inprod{v}{z-x}+R(z)+\frac{1}{\alpha}\cD_\phi(z,x)\right\}.
\]
For any given $v\in\R^n$ and $\alpha>0$, under the same assumption in Lemma \ref{well-def-prox-mapping}, ${\rm T}^{\phi}_{\alpha,R}(x,v)$ is also well defined. 

We end this subsection by introducing the concept of functional convergence over any compact set, as discussed in literature such as \cite{benaim2005stochastic,duchi2018stochastic,davis2020stochastic}. Given a sequence of mappings $f_n\in\cC(\mathbb{R}_+,\mathbb{R}^n)$, we say that $f_n$ converges to $f$ in $\cC(\R_+,\R^n)$ if, for any $T>0$, it holds that
\[
\lim_{n\rightarrow\infty}\sup_{t\in[0,T]}\|f_n(t)-f(t)\|=0.
\]

\section{Bregman differential inclusion}
In this section, we investigate the Bregman-type differential inclusion and its discrete approximation, which coincides with the iterative Bregman subgradient methods. The analysis tools employed in this section draw inspiration from various works on stochastic approximation, such as \cite{benaim2005stochastic, borkar2009stochastic, duchi2018stochastic, davis2020stochastic,ruszczynski2020convergence}. To establish the convergence of the discrete sequence $\{x_k\}$, a key idea is to show that the linear interpolation of the sequence $\{x_k\}$ defined by
\begin{equation}
{z(t)}:=x_k+\frac{t-\lambda_\eta(k)}{\eta_k}(x_{k+1}-x_k),\;t\in[\lambda_\eta(k),\lambda_\eta(k+1))
\label{Eq:x(t)_def}
\end{equation}
is a perturbed solution \cite{benaim2005stochastic} to the associated differential inclusion, where $\{\eta_k\}$ serves as the stepsize in the subgradient methods. However, due to the non-Euclidean nature of Bregman subgradient methods, this methodology requires modification. Given a kernel function $\phi$ and a general set-valued mapping $\mathcal{H}$, we consider the following general differential inclusion: 
\begin{equation}
    \label{Eq:Breg_DI}
    \frac{d}{dt}\nabla\phi(x(t))\in-\mathcal{H}(x(t)), \text{ for almost all $t\geq0$.}
\end{equation}
Any absolutely continuous solution to \eqref{Eq:Breg_DI} is termed a trajectory of \eqref{Eq:Breg_DI}. The stable set of \eqref{Eq:Breg_DI} is defined as 
\begin{equation}
\label{eq:def_stable_set}
\cH^{-1}(0):=\{x\in\R^n:0\in\cH(x)\}.
\end{equation}
{For any $\cC^2$ convex function $\phi$}, the differential inclusion \eqref{Eq:Breg_DI} can be interpreted as a gradient flow equipped with the Riemannian metric induced by $\inner{\cdot,\cdot}_{\nabla^2\phi(x)}$, as demonstrated in works such as \cite{bolte2003barrier,alvarez2004hessian}. The corresponding discrete scheme of \eqref{Eq:Breg_DI} is given by
\begin{equation}
   \label{Eq:Breg_general_iterative}  \nabla\phi(x_{k+1})=\nabla\phi(x_k)-\eta_k(d_k+\xi_k),
\end{equation}
where $d_k$ is an evaluation of $\cH(x_k)$ with possible inexactness, and $\xi_k$ is the stochastic noise. This formulation, referred to as the general Bregman-type method, is notable for the versatile choices of $\phi$ and $\cH$. For the Bregman counterpart of the interpolated process \eqref{Eq:x(t)_def} in the Euclidean setting, we introduce the linear interpolation for the dual sequence $\{\nabla\phi(x_k)\}$:
\begin{equation}
 \label{Eq:interpolation}
 x(t):=\nabla\phi^*\left(\nabla\phi(x_k)+\frac{t-\lambda_\eta(k)}{\eta_k}(\nabla\phi(x_{k+1})-\nabla\phi(x_k))\right),\;t\in[\lambda_\eta(k),\lambda_\eta(k+1)). 
\end{equation}
If $\phi^*\in\cC^1(\R^n)$, then \eqref{Eq:interpolation} is well defined. Let $x^t(\cdot)$ denote the time-shifted curve of the interpolated process, i.e. $x^t(\cdot)=x(t+\cdot)$. We make the following assumptions on \eqref{Eq:Breg_DI} and \eqref{Eq:Breg_general_iterative} to ensure that the iterative sequence generated by \eqref{Eq:Breg_general_iterative} tracks a trajectory of \eqref{Eq:Breg_DI} asymptotically.
\begin{assumption}
\label{Assumption: DI}
    \begin{enumerate}
        \item {$\phi$ is a supercoecive Legendre kernel function over $\R^n$}, and $\nabla\phi$ is differentiable almost everywhere.
        \item The sequences $\{x_k\}$, $\{\nabla\phi(x_k)\}$ and $\{d_k\}$ are uniformly bounded.
        \item The stepsize $\{\eta_k\}$ satisfies $\sum_{k=0}^\infty\eta_k=\infty$ and $\lim_{k\rightarrow\infty}\eta_k=0$.
        \item For any $T>0$, the noise sequence $\{\xi_k\}$ satisfies 
        \begin{equation}
            \lim_{s\rightarrow\infty}\sup_{s\leq i\leq\Lambda_\eta(\lambda_\eta(s)+T)}\norm{\sum_{k=s}^i\eta_k\xi_k}=0.
            \label{Eq:noise_cond}
        \end{equation}       
        \item The set-valued mapping $\mathcal{H}$ has a closed graph. Additionally, for any unbounded increasing sequence $\{k_j\}$ such that $\{x_{k_j}\}$ converges to $\bar x$, it holds that 
        \begin{equation}
            \lim_{N\rightarrow\infty}{\rm dist}\left(\frac{1}{N}\sum_{j=1}^Nd_{k_j},\mathcal{H}(\bar x)\right)=0.
        \end{equation}
    \end{enumerate}
\end{assumption}
Here are some remarks on Assumption \ref{Assumption: DI}. 
\begin{remark}
    \label{rmk:assumption_DI}
    \begin{enumerate}
        \item {If the Legendre function $\phi$ is supercoecive, i.e. $\lim_{\|u\|\rightarrow\infty}\frac{\phi(u)}{\|u\|}=\infty$, then by \cite[Theorem 26.5,\,Corollary 13.3.1]{rockafellar1997convex}, $\phi^*\in\cC^1(\R^n)$ is strictly convex, and $(\nabla\phi)^{-1}=\nabla\phi^*$.}
        \item Uniform boundedness of either $\{x_k\}$ or $\{\nabla\phi(x_k)\}$ may suffice under mild conditions. For example, if $\phi$ is locally strongly convex, then the uniform boundedness of $\{\nabla\phi(x_k)\}$ implies the uniform boundedness of $\{x_k\}$. Conversely, if $\nabla\phi$ is locally Lipschitz continuous, then the uniform boundedness of $\{x_k\}$ leads to the uniform boundedness of $\{\nabla\phi(x_k)\}$.
        \item For a martingale difference noise sequence $\{\xi_k\}$,  i.e. $\E[\xi_{k+1}|\cF_k]=0$ holds  almost surely,  as shown in \cite{benaim2005stochastic}, \eqref{Eq:noise_cond} can be ensured almost surely under one of the following conditions:
        \begin{enumerate}
            \item Uniform boundedness of $\{\xi_k\}$ with $\eta_k=o\left(\frac{1}{\log k}\right)$ as utilized in \cite{castera2021inertial,xiao2023adam,xiao2023convergence}.
            \item Variance-bounded $\{\xi_k\}$, i.e. $\E[\norm{\xi_{k+1}}^2|\cF_k]\leq\sigma^2$ for some $\sigma>0$, and square-summable $\{\eta_k\}$, i.e. $\sum_{k=0}^\infty\eta_k^2<\infty$, as employed in \cite{davis2020stochastic}.
        \end{enumerate}
        \item Assumption \ref{Assumption: DI}(5) describes how $d_k$ approximates $\cH(x_k)$. This assumption is quite mild. A sufficient condition for it hold is that $d_k\in\cH^{\delta_k}(x_k)$ with $\lim_{k\to\infty}\delta_k=0$. 
    \end{enumerate}
\end{remark}

{
The following theorem demonstrates that the discrete sequence generated by the general Bregman-type method \eqref{Eq:Breg_general_iterative} asymptotically tracks a trajectory of the differential inclusion \eqref{Eq:Breg_DI}. Our key technical novelty, compared to previous works on stochastic approximation, especially \cite{benaim2005stochastic}, is in establishing convergence under weaker assumptions on the kernel regularity. In \cite{benaim2005stochastic}, the analysis focuses on differential inclusions of the form \(\dot{x} \in -\mathcal{G}(x)\) and discrete updates are explicitly defined by \(x_{k+1}=x_k - \eta_k g_k\) with \(g_k \in \mathcal{G}(x_k)\). Directly applying their results to our Bregman-type updates would require verifying convergence to a set-valued mapping \(\widetilde{\mathcal{H}}\), specifically establishing:
\[
\mathrm{dist}\left(\frac{x_{k} - \nabla\phi^*(\nabla\phi(x_k)-\eta_k d_k)}{\eta_k}, \widetilde{\mathcal{H}}(x_k)\right) \to 0.
\]
When the kernel \(\phi\) is \(\mathcal{C}^2(\mathbb{R}^n)\), this convergence can indeed be guaranteed, as we can explicitly choose \(\widetilde{\mathcal{H}}(x)=(\nabla^2\phi(x))^{-1}\mathcal{H}(x)\). However, our analysis relaxes this requirement significantly, assuming only that \(\nabla\phi\) is differentiable almost everywhere. Under this weaker condition, such a direct equivalence no longer holds. To overcome this issue, we adopt a carefully designed interpolation scheme for the dual sequence, as defined in \eqref{Eq:interpolation}. By employing this novel linear interpolation approach {on the sequence $\{\nabla\phi(x_k)\}$}, we successfully establish convergence without the need for twice continuous differentiability of \(\phi\). The complete details of our approach and proofs are provided in Appendix A.}

\begin{theorem}\label{subsequential-thm}
Suppose Assumption \ref{Assumption: DI} holds. Then, for any sequence $\{\tau_k\}_{k=1}^\infty\subset\R_+$, the set of shifted sequence $\{\nabla\phi(x^{\tau_k}(\cdot))\}_{k=1}^\infty$ is relatively compact in $\mathcal{C}(\R_+,\R^n)$. If $\lim_{k\rightarrow\infty}\tau_k=\infty$, then any cluster point $\bar y(\cdot)$ of $\{\nabla\phi(x^{\tau_k}(\cdot))\}_{k=1}^\infty$ belongs to $\mathcal{C}(\R+,\R^n)$. Define $\bar x(\cdot):=\nabla\phi^*(\bar y(\cdot))$. Then there exists a measurable $\bar d(t)\in\mathcal{H}(\bar x(t))$ satisfying 
\begin{equation}
\nabla\phi(\bar x(t))=\nabla\phi(\bar x(0))-\int_{0}^t\bar d(\tau)d\tau\quad for\;all\;t\geq0.
\label{Eq:DI_int_form}
\end{equation}
Equivalently, we have 
\[
\frac{d}{dt}\nabla\phi(\bar x(t))\in-\mathcal{H}(\bar x(t)),\;for\;almost\;all\;t\geq0.
\]
\end{theorem}

The following assumption ensures that the trajectory subsequentially converges to the stable set of \eqref{Eq:Breg_DI}, and the Lyapunov function values converge.
\begin{assumption}\label{assumption_Sard_Lyapunov}
 There exists a continuous function $\Psi:\R^n\to\R$, such that the following conditions hold:
 \begin{enumerate}
     \item (Weak Morse-Sard). The set 
     {$\{\Psi(x): \mbox{for $x$ such that $0\in\cH(x)$} \}$}  has empty interior in $\R$.
     \item (Lyapunov function). $\Psi$ is lower bounded, i.e. $\liminf_{x\in\R^n}\Psi(x)>-\infty$. For any trajectory $z(t)$ of the differential inclusion \eqref{Eq:Breg_DI} with $z(0)\notin\cH^{-1}(0)$,  there exists $T > 0$ such that
     \[
     \Psi(z(T))<\sup_{t\in[0,T]}\Psi(z(t))\leq\Psi(z(0)).
     \]
 \end{enumerate}
\end{assumption}
{We say that a continuous function $\Psi$, satisfying Assumption \ref{assumption_Sard_Lyapunov}.2, is a Lyapunov function for the differential inclusion \eqref{Eq:Breg_DI} with a stable set $\cH^{-1}({0})$.}

Now, we are ready to present the main theorem in this section.
\begin{theorem}\label{convergence-thm-func-val}
Suppose Assumptions \ref{Assumption: DI} and \ref{assumption_Sard_Lyapunov} hold. Then any cluster point of $\{x_k\}$ lies in $\mathcal{H}^{-1}(0)$ and the function values $\{\Psi(x_k)\}_{k\geq1}$ converge.
\end{theorem}
\begin{proof}
{Since $\{x_k\}$ is bounded, let $x^*$ be any cluster point of $\{x_k\}$, and ${\lim_{i\rightarrow\infty}}{x_{k_i}}=x^*$. By Theorem \ref{subsequential-thm}, up to a subsequence, $\nabla\phi(x^{\lambda_\eta(k_i)}(\cdot))\rightarrow\nabla\phi(\bar x(\cdot))$ for some $\bar x(\cdot)\in\cC(\R_+,\R^n)$. Note that $x_{{k_i}}=x^{\lambda_\eta(k_i)}(0)$, so we have 
\[
\begin{aligned}
\norm{\bar x(0)-x^*}\leq&\limsup_{i\rightarrow\infty}\left(\norm{\bar x(0)-x^{\lambda_\eta(k_i)}(0)}+\norm{x^{\lambda_\eta(k_i)}(0)-x_{k_i}}+\norm{x_{k_i}-x^*}\right)\\
=&\limsup_{i\rightarrow\infty}\left(\norm{\nabla\phi^*(\nabla\phi(\bar x(0)))-\nabla\phi^*(\nabla\phi(x^{\lambda_\eta(k_i)}(0)))}+\norm{x^{\lambda_\eta(k_i)}(0)-x_{k_i}}+\norm{x_{k_i}-x^*}\right)\\
=&0,
\end{aligned}
\]
where the last equality comes from the continuity of $\nabla\phi^*$.} Hence $\bar x(0)=x^*$. Suppose $x^*\notin \mathcal{H}^{-1}(0)$, then by Assumption \ref{assumption_Sard_Lyapunov}, there exists $T>0$, such that 
\[
\Psi(\bar x(T))<\sup_{t\in[0,T]}\Psi(\bar x(t))\leq\Psi(x^*).
\]
On the other hand, by Lemma \ref{prop:nonescape}, $\Psi(x(t))$ converges {as $t\to\infty$}. Therefore, we obtain
\[
\Psi(\bar x(T))={\lim_{i\rightarrow\infty}}\Psi(x^{\lambda_\eta(k_i)}(T))=
{\lim_{i\rightarrow\infty}}
\Psi(x({\lambda_\eta(k_i)}+T))=\lim_{t\rightarrow\infty}\Psi(x(t))=\Psi(x^*), 
\]
which leads to a contradiction. Therefore, $0\in \mathcal{H}(x^*)$. This completes the proof.
\end{proof}

The results in Theorem \ref{convergence-thm-func-val} illustrate that the sequence $\{x_k\}$ finds the stable set of \eqref{Eq:Breg_DI}, and the function values of Lyapunov function converge. Due to the versatility in selecting the set-valued mapping $\cH$, we will later show that if $\cH$ is chosen as the conservative field of $f$,  the sequence generated by the Bregman-type method \eqref{Eq:Breg_general_iterative} converges to a stationary point of the optimization problem \eqref{min-prob}.

{
In the remainder of this section, we aim to establish the \textbf{global stability} of the framework \eqref{Eq:Breg_general_iterative} under the 
noise sequence $\{\xi_{k}\}$ correspond to {\bf random reshuffling}. To that end, we impose the following assumptions on the stepsize $\{\eta_k\}$ and the noise $\{\xi_{k}\}$. 

\begin{assumption}
    \label{Assumption_Reshuffling}
    There exists an integer $N> 0$ such that 
\begin{enumerate}
    \item For any nonnegative integers $i,j < N$, it holds that $\eta_{kN+i} = \eta_{kN+j}$ for any $k\in \mathbb{N}_+$. 
    \item The sequence $\{\xi_k\}$ is uniformly bounded. Moreover, for any $j \in \mathbb{N}_+$, almost surely, it holds that $\sum_{{k}= jN}^{(j+1)N-1} \xi_{k+1} = 0$. 
\end{enumerate}
\end{assumption}
Here we make some remarks on Assumptions \ref{Assumption_Reshuffling}. Assumption \ref{Assumption_Reshuffling}(1) indicates that the stepsizes $\{\eta_k\}$ remain constant within each epoch, which is a standard setting in neural network training tasks.
Moreover, Assumption \ref{Assumption_Reshuffling}(2) encodes the essence of random reshuffling in the stochastic subgradient computation of $\{d_k\}$ in \eqref{Eq:Breg_general_iterative}. Section~\ref{sec:application} provides a concrete example in the context of finite-sum minimization, illustrating why this assumption is reasonable.

\begin{lemma}
    \label{Le_controlled_noise_RR}
    Suppose Assumption \ref{Assumption_Reshuffling} holds for the sequence of noises $\{\xi_{k}\}$ and stepsizes $\{\eta_k\}$. Moreover, the Lyapunov function $\Psi$ associated with $\cH$ is coercive. Then for any $\varepsilon > 0$ and $T>0$, there exists $\eta_{\varepsilon} > 0$ such that for any $\{\eta_k\}$ satisfying $\limsup_{k\to +\infty} \eta_k \leq \eta_{\varepsilon}$, it holds that 
    \begin{equation}
        \limsup_{s\to +\infty} \sup_{s \leq i\leq {\Lambda_\eta(\lambda_\eta(s) + T)} }  \norm{ \sum_{k = s}^i \eta_k \xi_{k+1}} \leq \varepsilon.  
    \end{equation}
\end{lemma}

We now state a theorem establishing the global stability of \eqref{Eq:Breg_general_iterative} under non-diminishing stepsizes  $\{\eta_k\}$. 
\begin{theorem}
    \label{The_convergence_Nondiminishing_RR}
Suppose Assumptions \ref{assumption_Sard_Lyapunov} and \ref{Assumption_Reshuffling} hold.  Additionally, assume $\phi$ is strongly convex and Lipschitz smooth, and that $d_k\in\cH^{\delta_k}(x_k)$ with $\lim_{k\to\infty}\delta_k=0$.  Then for any $\varepsilon > 0$, there exists $\eta_{\max}(\varepsilon) > 0$ such that for any $\{\eta_k\}$ satisfying $\limsup_{k\to +\infty} \eta_k \leq \eta_{\max}(\varepsilon)$, the following holds: 
\begin{equation}
    \limsup_{k\to +\infty}\;\mathrm{dist}\left( x_k, \{x \in \R^n: 0\in \cH(x)\}  \right) \leq \varepsilon. 
\end{equation}
Consequently, the sequence $\{x_k\}$ is uniformly bounded.
\end{theorem}
Theorem \ref{The_convergence_Nondiminishing_RR} implies that as long as the stepsizes $\{\eta_k\}$ are sufficiently small, the sequence $\{x_k\}$ remains stable. Notably, when random reshuffling noise is present, together with other essential mild conditions, the uniform boundedness of the iterates can be rigorously established, thereby justifying the boundedness assumption in Assumption \ref{Assumption: DI}.
}

\section{Applications}
\label{sec:application}
Based on the framework of the general Bregman-type method as outlined in \eqref{Eq:Breg_DI}, in this section, we consider three specific types of stochastic Bregman subgradient methods by choosing different types of kernel function $\phi$ and  set-valued mapping $\cH$. In the first two parts, we consider vanilla and single timescale  momentum based stochastic Bregman subgradient methods for unconstrained optimization problems. Subsequently, we extend our methods to the stochastic Bregman proximal subgradient method for solving constrained composite optimization problems.

\subsection{Stochastic Bregman subgradient method}
In this subsection, we consider the following stochastic Bregman subgradient update scheme:
\begin{equation}
    \label{Eq:BGD}
    \tag{SBG}
    \begin{aligned}x_{k+1}&\approx\arg\min_{x\in\R^n}\;\left\{\inner{g_k,x-x_k}+\frac{1}{\eta_k}\cD_\phi(x,x_k)\right\},\\
    \text{s.t. }& {\text{ $x_{k+1}$ is a $\nu_k$-optimal solution of the subproblem.} }
    \end{aligned}
\end{equation}
where $g_k=d_k+\xi_k$, $d_k\in D^{\delta_k}_f(x_k)$, and $\xi_k$ is the stochastic noise. The associated differential inclusion is 
\begin{equation}
    \frac{d}{dt}\nabla\phi(x(t))\in-D_f(x(t)).
    \label{Eq:DI_BGD}
\end{equation}
{Recall the definition of $\nu$-optimal solution in Section \ref{sec:notation}, $x_{k+1}$ is a $\nu_k$-optimal solution of the subproblem, if and only if $\|g_k+\frac{1}{\eta_k}\left(\nabla\phi(x_{k+1})-\nabla\phi(x_k)\right)\|\leq\nu_k$.} Given the allowance for inexact solutions in the SBG framework, we illustrate that a kernel Hessian preconditioned subgradient method fits within our framework, akin to the approach in the recently proposed ABPG in \cite{takahashi2023approximate}. The concept of Hessian preconditioning has also been examined in continuous settings as seen in \cite{alvarez2004hessian,bolte2003barrier}. Assuming the absence of the nonsmooth term in ABPG, the existence and nonsingularity of $\nabla^2\phi$ everywhere, and that $\sup_k\norm{(\nabla^2\phi(x_k))^{-1}}\leq c$ for some $c>0$, the ABPG updates scheme in \cite{takahashi2023approximate} is given by
\begin{equation}
\label{Eq:iSBG}
x_k^+=x_k-\eta_k(\nabla^2\phi(x_k))^{-1}g_k.
\end{equation}
This yields
\[
\begin{aligned}
&\nabla\phi(x_k^+)-(\nabla\phi(x_k)-\eta_kg_k)\\
=&\nabla\phi(x_k-\eta_k(\nabla^2\phi(x_k))^{-1}g_k)-(\nabla\phi(x_k)-\eta_kg_k)\\
=&\nabla\phi(x_k)-\eta_k\nabla^2\phi(x_k)(\nabla^2\phi(x_k))^{-1}g_k+o(\eta_k)-(\nabla\phi(x_k)-\eta_kg_k)\\
=&o(\eta_k),
\end{aligned}
\]
indicating that $\lim_{k\rightarrow\infty}\norm{g_k+\frac{\nabla\phi(x^+_k)-\nabla\phi(x_k)}{\eta_k}}=0$, which implies that $x^+_k$ is an approximate solution to the SBG subproblem. 


We make the following assumptions on \eqref{Eq:BGD}. 
\begin{assumption}
    \label{assumption:SBGD}
    \begin{enumerate}
        \item $\phi$ is a supercoecive Legendre kernel function over $\R^n$, and $\nabla\phi$ is differentiable almost everywhere. {Moreover, for any absolutely continuous mapping $z(\cdot)\in\cC(\R_+,\R^n)$, $\nabla^2\phi(z(s))$ is positive definite for almost all $s\geq0$.}
        \item The sequences $\{x_k\}$, $\{\nabla\phi(x_k)\}$ and $\{d_k\}$ are uniformly bounded almost surely.
        \item $\{\xi_k\}$ is a martingale difference noise, i.e. $\E[\xi_{k+1}|\cF_k]=0$ holds almost surely. $\sum_{k=0}^\infty \eta_k = \infty$, and the stepsize and noise satisfy one of the following two conditions:
        \begin{enumerate}
            \item $\{\xi_k\}$ is uniformly bounded and $\eta_k=o\left(\frac{1}{\log k}\right)$.
            \item $\{\xi_k\}$ has bounded variance, i.e. $\E[\norm{\xi_{k+1}}^2|\cF_k]\leq\sigma^2<\infty$, and $\sum_{k=0}^\infty\eta_k^2<\infty$.
        \end{enumerate}      
        \item $\lim_{k\rightarrow\infty}\delta_k=0$, and $\lim_{k\rightarrow\infty}\nu_k=0$.
    \end{enumerate}
\end{assumption}

{
Now, we make some remarks regarding Assumption \ref{assumption:SBGD}. Assumption \ref{assumption:SBGD}(1) and (3) are directly adapted from Assumption \ref{Assumption: DI}(1) and (4), respectively. Assumption \ref{assumption:SBGD}(4) characterizes how $d_k$ approximates $D_f(x_k)$ and quantifies the inexactness in solving the subproblem. Assumption \ref{assumption:SBGD}(2) can be ensured by global stability under randomly reshuffling noise $\{\xi_k\}$, as demonstrated by Theorem \ref{The_convergence_Nondiminishing_RR}. We elaborate on these points further below.

Consider the case where $f(x) = \frac{1}{N}\sum_{i=1}^N f_i(x)$ being coercive, each $f_i$ is path-differentiable and admits a convex-valued local bounded conservative field $D_{f_i}(x)$. At each iteration, we sample an index $i_k$ from $[N] := \{1, 2, \dots, N\}$ such that for each epoch $j \geq 0$, $\{i_k : jN \leq k < (j+1)N\} = [N]$ holds. We then compute $d_k \in D_{f_{i_k}}(x_k)$. This setup corresponds to the standard random reshuffling framework. For each $k\geq 0$, let $j_k \geq 0$ be such that $k \in [j_kN, (j_k+1)N)$, it follows that $d_k \in D_f(x_k) + \xi_{k+1}$, where $\xi_{k+1}$ represents the evaluation noise. Defining $\delta_k = 2N\sum_{j_k N \leq l < (j_k +1)N} \norm{x_{l+1} - x_l}$, 
we have $d_k \in D_{f_{i_k}}(x_k) \subseteq D_{f_{i_k}}^{\delta_k / (2N)}(x_{j_kN})$. Consequently,  
\begin{equation*}
\frac{1}{N} \sum_{j_kN \leq l < (j_k+1)N} d_l \in \frac{1}{N} \sum_{j_kN \leq l < (j_k+1)N} D_{f_{i_l}}^{\delta_k / (2N)}(x_{j_kN}) \subseteq D_f^{\delta_k / 2}(x_{j_kN}) \subseteq D_f^{\delta_k}(x_k).
\end{equation*}
Furthermore, it holds that $\frac{1}{N} \sum_{jN \leq k < (j+1)N} \xi_{k+1} = 0$. Therefore, we can conclude that Assumption \ref{Assumption_Reshuffling}(2) corresponds to the setting where the stochastic subgradients $\{d_k\}$ are generated by the random reshuffling sampling technique. Therefore, applying Theorem \ref{The_convergence_Nondiminishing_RR}, we can obtain the uniform boundedness of $\{x_k\}$. {Since \(D_f\) is locally bounded and \(\phi\) is locally Lipschitz (implying that \(\nabla\phi\) is locally bounded), the boundedness of the sequence \(\{x_k\}\) guarantees that both \(\{\nabla\phi(x_k)\}\) and \(\{d_k\}\) are uniformly bounded (see, e.g., \cite[Proposition 5.15]{RockWets98}).}

To ensure the convergence of \eqref{Eq:BGD}, we make the following assumptions on $f$.
}

\begin{assumption}
    \label{assumption:Lyapunov_BGD}
    \begin{enumerate}
        \item $f$ is lower bounded, i.e. $\liminf_{x\in\R^n}f(x)>-\infty$. Moreover, $f$ is a potential function that admits $D_f$ as its convex-valued conservative field. 
        \item The critical value set $\{f(x):\;0\in D_f(x)\}$ has empty interior in $\R$.
    \end{enumerate}
\end{assumption}
{Assumption \ref{assumption:Lyapunov_BGD} is a specific instance of Assumption \ref{assumption_Sard_Lyapunov}. In particular, $f$ is the Lyapunov function and Assumption \ref{assumption:Lyapunov_BGD}(2) requires a weak Sard condition for \(f\). Recall that the classical Sard theorem for differentiable functions in \(\mathbb{R}^n\) states that the set of critical values has Lebesgue measure zero.}

We have the following two propositions.
\begin{lemma}
    \label{le:SBGD}
    Suppose Assumptions \ref{assumption:SBGD} and \ref{assumption:Lyapunov_BGD} hold. For any $d^e_k$ such that $\norm{d^e_k}\leq\nu_k$, and any increasing sequence $\{k_j\}$ such that $\{x_{k_j}\}$ converges to $\bar x$, it holds that 
    \[
    \lim_{N\rightarrow\infty}\dist\left(\frac{1}{N}\sum_{j=1}^N(d_{k_j}+d^e_{k_j}),D_f(\bar x)\right)=0.
    \]
\end{lemma}
\begin{proof}
    By the inexact condition in \eqref{Eq:BGD}, it follows that there exists $d^e_k$, such that $\norm{d^e_k}\leq\nu_k$ and 
    \[
    \nabla\phi(x_{k+1})=\nabla\phi(x_k)-\eta_k\left(d_k+d^e_k+\xi_k\right).
    \]
    Define $\tilde d_k:=d_k+d^e_k\in D_f^{\tilde\delta_k}(x_k)$, where $\tilde\delta_k=\delta_k+\nu_k$. Note that $\lim_{k\rightarrow\infty}\tilde\delta_k=0$. Since $D_f$ has a closed graph, then for any $\{x_{k_j}\}$ converging to $\bar x$, it holds that $\lim_{j\rightarrow\infty}\dist\left(\tilde d_{k_j},D_f(\bar x)\right)=0$. Note that $D_f(\bar x)$ is a convex set, by Jensen's inequality, we have
    \[
    \lim_{N\rightarrow\infty}\dist\left(\frac{1}{N}\sum_{j=1}^N\tilde d_{k_j},D_f(\bar x)\right)=\lim_{N\rightarrow\infty}\frac{1}{N}\sum_{j=1}^N\dist\left(\tilde d_{k_j},D_f(\bar x)\right)=0.
    \]
    This completes the proof.
\end{proof}

\begin{proposition}
    \label{prop:Lyapunov_BGD}
    Suppose Assumptions \ref{assumption:SBGD} and \ref{assumption:Lyapunov_BGD} hold. Then $f$ is a Lyapunov function for the differential inclusion \eqref{Eq:DI_BGD} with the stable set $\left\{x\in\R^n:\;0\in D_f(x)\right\}$.
\end{proposition}
\begin{proof}
Consider any trajectory $z(\cdot)$ for the differential inclusion \eqref{Eq:DI_BGD} with $0\notin D_f(z(0))$. We have that for almost all $s\geq0$,
\[
\begin{aligned}
\frac{d}{ds}f(z(s))=\inner{D_f(z(s)),\dot z(s)}\ni-\inner{\nabla^2\phi(z(s))\dot z(s),\dot z(s)}
\end{aligned}
\]
Therefore, for any $t\geq0$, it holds that
\[
f(z(t))-f(z(0))=-\int_0^t\inner{\nabla^2\phi(z(s))\dot z(s),\dot z(s)}ds\leq-\int_0^t\lambda_{\min}\left(\nabla^2\phi(z(s))\right)\norm{\dot z(s)}^2ds\leq0.
\]
We now prove the required result by contradiction. Suppose for any $t\geq0$, $f(z(t))=f(z(0))$, then, we have $\lambda_{\min}\left(\nabla^2\phi(z(s))\right)\norm{\dot z(s)}^2=0$ for almost all $s\geq0$. Since $\lambda_{\min}(\nabla^2\phi(z(\cdot)))>0$ almost everywhere in $\R_+$, then $\dot z(s)=0$ for almost all $s\geq0$. Since $z(\cdot)$ is absolutely continuous, therefore, $z(t)\equiv z(0)$ for any $t\geq0$. Then, $0=\frac{d}{dt}\nabla\phi(z(t))\in-D_f(z(t))=-D_f(z(0))$. This is contradictory to the fact that $z(0)$ is not a $D_f$-stationary point of $f$. Therefore, there exists $T>0$, such that $f(z(T))<\sup_{t\in[0,T]}f(z(t))\leq f(z(0))$. This completes the proof.
\end{proof}
By Lemma \ref{le:SBGD}, Proposition \ref{prop:Lyapunov_BGD} and Theorem \ref{convergence-thm-func-val}, we can directly derive the following convergence results for \eqref{Eq:BGD}.
\begin{theorem}
    \label{thm:BGD}
    Suppose Assumptions \ref{assumption:SBGD} and \ref{assumption:Lyapunov_BGD} hold. {Then almost surely,} any cluster point of $\{x_k\}$ generated by \eqref{Eq:BGD} is a $D_f$-stationary point and the function values $\{f(x_k)\}$ converge. 
\end{theorem}



\subsection{Momentum based stochastic Bregman subgradient method}
{In this section, we introduce a momentum-based stochastic Bregman subgradient method (MSBG). The momentum technique is widely adopted in practice, particularly for training neural networks using stochastic first-order methods such as SGDM, Adam, and AdamW, primarily due to its empirical effectiveness. Nevertheless, a complete theoretical understanding of momentum's benefits, even within smooth optimization contexts, remains elusive. While certain studies have attempted to interpret momentum through generalization perspectives such as margin analysis and algorithmic stability (\cite{jelassitowards,ramezani2024generalization}), such analysis is beyond the scope of pure optimization, and thus we omit further discussion of these issues here.

From an optimization perspective, several studies (\cite{defazio2020understanding,liu2020improved,guo2021novel}) highlight that momentum can accelerate convergence by canceling out some impact of stochastic gradient noise. Intuitively, momentum reduces variance in the gradient estimates, effectively reducing the noise that often dominates convergence behavior, particularly in high-noise settings typical of neural network training. Motivated by these observations, we also incorporate momentum into our proposed Bregman subgradient method (MSBG), leveraging its variance-reducing benefits to potentially enhance convergence performance in practice. }

For a chosen kernel function $\varphi: \mathbb{R}^n \rightarrow \mathbb{R}$, the momentum based update scheme is given as follows:
\begin{equation}
    \label{Eq:MSBG_update}
    \tag{MSBG}
    \left\{\begin{aligned}
        x_{k+1}&\approx\argmin_{x\in\R^n}\left\{\inner{m_k,x-x_k}+\frac{1}{\eta_k}\cD_{\varphi}(x,x_k)\right\}\\
        \text{s.t. }&{\text{ $x_{k+1}$ is a $\nu_k$-optimal solution of the subproblem.} }\\
        m_{k+1}&=m_k-\theta_kP(x_k)(m_k-g_k),
    \end{aligned}\right.
\end{equation}
where $g_k = d_k + \xi_k$, $d_k \in D^{\delta_k}_f(x_k)$, and $P(x_k)\in\R^{n\times n}$ denotes a preconditioning matrix. Similar to \eqref{Eq:iSBG}, the MSBG subproblem can also be solved in an inexact manner by adopting a preconditioned subgradient strategy as shown in \eqref{Eq:iSBG}. For the ease of presentation, we omit the discussion. 
{Our MSBG method is a 
single timescale method in the sense} that the stepsize $\eta_k$ for the primal variable and the stepsize $\theta_k$ for the momentum decay at the same rate.

We make the following assumptions on \eqref{Eq:MSBG_update}.

\begin{assumption}
\label{assumption_single}
\begin{enumerate}
    \item {$\varphi\in\cC^2(\R^n)$ is a supercoecive Legendre kernel function over $\R^n$, and $\nabla^2\varphi(\cdot)$ is positive definite everywhere. Moreover, $P(\cdot)=(\nabla^2\varphi(\cdot))^{-1}:\R^n\rightarrow\R^{n\times n}$.}
    \item The sequences $\{x_k\}$, $\{\nabla\phi(x_k)\}$, $\{d_k\}$, and $\{m_k\}$ are uniformly bounded almost surely.
   \item $\{\xi_k\}$ is a martingale difference noise, $\sum_{k=0}^\infty\eta_k=\infty$, and the stepsize and noise satisfy one of the following two conditions:
        \begin{enumerate}
            \item $\{\xi_k\}$ is uniformly bounded and $\eta_k=o\left(\frac{1}{\log k}\right)$.
            \item $\{\xi_k\}$ has bounded variance, i.e. $\E[\norm{\xi_{k+1}}^2|\cF_k]\leq\sigma^2<\infty$, and $\sum_{k=0}^\infty\eta_k^2<\infty$.
        \end{enumerate} 
    \item  $\lim_{k\rightarrow\infty}\delta_k=0$ and $\lim_{k\rightarrow\infty}\nu_k=0$.
    \item There exists a positive $\tau$ such that $\lim_{k\rightarrow\infty}\frac{\theta_k}{\eta_k}=\tau$.
\end{enumerate}
\end{assumption}
{
We make some remarks on Assumption \ref{assumption_single}. Assumption \ref{assumption_single}(1) and (3) are directly adapted from Assumption \ref{Assumption: DI}(1) and (4), respectively. Assumption \ref{assumption_single}(4) characterizes how $d_k$ approximates $D_f(x_k)$ and quantifies the inexactness in solving the subproblem. 
We note that Assumption \ref{assumption_single}(2) can be ensured by the global stability when the noise sequence $\{\xi_k\}$ corresponds to randomly reshuffling noises, as demonstrated by Theorem \ref{The_convergence_Nondiminishing_RR}. Assumption \ref{assumption_single}(5) assumes that the stepsizes $\{\eta_k\}$ and momentum parameters $\{\theta_k\}$ in the framework \eqref{alg:MSBPG} are single-timescale, i.e. they decrease at the same rate. In practice, $\tau$ can be preset.}

Consider the following differential inclusion,
\begin{equation}
    \begin{aligned}
        \frac{d}{dt}\left[\begin{array}{c}
            \nabla\varphi(x(t)) \\
             m(t)
             \end{array}\right]\in-\left[\begin{array}{c}
                m(t) \\
                \tau(\nabla^2\varphi(x(t)))^{-1}(m(t)-D_f(x(t)))
             \end{array}\right],\text{ for almost all $t\geq0$.}
    \end{aligned}
    \label{Eq_ode_single_timescale}
\end{equation}
Define
\[
\phi(x,m):=\varphi(x)+\frac{1}{2}\|m\|^2,\;\quad\mathcal{H}(x,m):=\left[\begin{array}{c}
                m \\
                \tau(\nabla^2\varphi(x))^{-1}(m-D_f(x) )
             \end{array}\right].
\]
Then, \eqref{Eq_ode_single_timescale} can be reformulated in the form of \eqref{Eq:Breg_DI} as:
\begin{equation}
    \label{Eq:Breg_DI_single_time}
    \frac{d}{dt}\nabla\phi(x(t),m(t))\in-\mathcal{H}(x(t),m(t)),\text{ for almost all $t\geq0$.}
\end{equation}
{The stable set of \eqref{Eq:Breg_DI_single_time} is given by $\cH^{-1}(0)=\{(x,m):\;m=0,0\in D_f(x)\}$. Based on the differential inclusion \eqref{Eq:Breg_DI_single_time}, \eqref{Eq:MSBG_update} can be reformulated in the form of \eqref{Eq:Breg_general_iterative} as:}
\begin{equation}
\left\{\begin{aligned}
    \nabla\varphi(x_{k+1})&=\nabla\varphi(x_k)-\eta_k(m_k+d^e_k)\\
    m_{k+1}&=m_k-\eta_k\cdot\frac{\theta_k}{\eta_k}(\nabla^2\varphi(x_k))^{-1}(m_k-d_k),
\end{aligned}\right.
    \label{Eq__discrete_single_timescale}
\end{equation}
where $\norm{d^e_k}\leq\nu_k$.

\begin{lemma}
    \label{le:asymp_gradient}
    Suppose Assumptions \ref{assumption_single} and \ref{assumption:Lyapunov_BGD} hold. Let $\{(x_k,m_k)\}$ be the sequence generated by \eqref{Eq:MSBG_update}, $d_{x,k}:=m_k+d^e_k$, and $d_{m,k}:=\frac{\theta_k}{\eta_k}(\nabla^2\varphi(x_k))^{-1}(m_k-d_k)$, where $\norm{d^e_k}\leq\nu_k$. For any increasing sequence $\{k_j\}$ such that $(x_{k_j},m_{k_j})$ converges to $(\bar x,\bar m)$, it holds that 
    \[
    \lim_{N\rightarrow\infty}\dist\left(\frac{1}{N}\sum_{j=1}^N(d_{x,k_j},d_{m,k_j}),\mathcal{H}(\bar x,\bar m)\right)=0.
    \]
\end{lemma}
\begin{proof}
By Assumption \ref{assumption_single}, it holds that
\[
\lim_{j\rightarrow\infty}\dist\left(m_{k_j}+d^e_{k_j},\bar m\right)\leq\lim_{j\rightarrow\infty}\dist(m_{k_j},\bar m)+\nu_{k_j}=0.
\]
{By Assumption \ref{assumption_single}.1, we have $(\nabla^2\varphi(x))^{-1}$ is continuous, and hence
\[
\lim_{j\rightarrow\infty}\dist\left((\nabla^2\varphi(x_{k_j}))^{-1}(m_{k_j}-d_{k_j}),(\nabla^2\varphi(\bar x))^{-1}(\bar m-D_f(\bar x))\right)=0.
\]}
Since $D_f(\bar x)$ is a compact set, and $\lim_{k\rightarrow\infty}\frac{\theta_k}{\eta_k}=\tau$, it holds that
\[
\lim_{j\rightarrow\infty}\dist\left(\frac{\theta_{k_j}}{\eta_{k_j}}(\nabla^2\varphi(x_{k_j}))^{-1}(m_{k_j}-d_{k_j}),\tau(\nabla^2\varphi(\bar x))^{-1}(\bar m-D_f(\bar x))\right)=0.  
\]
By the fact that $D_f(\bar x)$ is a convex set and Jensen's inequality, we have that 
\[
 \lim_{N\rightarrow\infty}\dist\left(\frac{1}{N}\sum_{j=1}^N(d_{x,k_j},d_{m,k_j}),\mathcal{H}(\bar m,\bar x)\right)\leq\lim_{N\rightarrow\infty}\frac{1}{N}\sum_{j=1}^N\dist\left((d_{x,k_j},d_{m,k_j}),\mathcal{H}(\bar m,\bar x)\right)=0.
\]
This completes the proof.
\end{proof}
\begin{proposition}
\label{prop:MSBG_Lyapunov}
    Suppose Assumptions \ref{assumption_single} and \ref{assumption:Lyapunov_BGD} hold. Then $h(x,m)=f(x)+\frac{1}{2\tau}\|m\|^2$ is a Lyapunov function for \eqref{Eq_ode_single_timescale} with the stable set $\cB:=\{(x,m)\in\R^n\times\R^n:\;m=0,0\in D_f(x)\}$.
\end{proposition}
\begin{proof}
Consider any trajectory $(x(t),m(t))$ for the differential inclusion \eqref{Eq_ode_single_timescale} with $(x(0),m(0))\notin\cB$. There exists measurable $d_f(s)\in D_f(x(s))$, such that for almost all $s\geq0$,
    \[
    \begin{aligned}
        &\frac{d}{ds}h(x(s),m(s))\\
        =&\inner{D_f(x(s)),\dot x(s)}+\inner{\frac{m(s)}{\tau},\dot m(s)}\\
        \ni&-\inner{d_f(s),(\nabla^2\varphi(x(s)))^{-1}m(s)}-\inner{m(s),(\nabla^2\varphi(x))^{-1}(m(s)-d_f(s))}\\
        =&-\inner{m(s),(\nabla^2\varphi(x(s)))^{-1}m(s)}.
    \end{aligned}
    \]
    Thus, for any $t\geq0$, $h(x(t), m(t)) \leq h(x(0), m(0))$. For any $(x(0),m(0))\notin\cB$, either $m(0) \neq 0$ or $m(0)=0$ and $0 \notin D_f(x(0))$. If $m(0) \neq 0$, then the continuity of $m(\cdot)$ ensures the existence of $T > 0$ and $\alpha > 0$ where $\|m(s)\| \geq \alpha$ for $s \in [0, T]$. Thus we have 
    \[
    h(x(T),m(T))-h(x(0),m(0))\leq-\int_{0}^T\inner{m(s),(\nabla^2\varphi(x(s)))^{-1}m(s)}ds<0.
    \]
    Now consider the case $m(0)=0$ and $0\notin D_f(x(0))$. By the outer semicontinuity of $D_f$ and Assumption \ref{assumption_single}.1, there exists $\tilde T>0$, such that for almost all $t\in[0,\tilde T]$, it holds that $0\notin(\nabla^2\varphi(x(t)))^{-1}D_f(x(t))$. Now suppose for {all} $t\geq0$, $h(x(t),m(t))=h(x(0),m(0))$, then we have $m(s)=0$ for almost all $s\geq0$. Since $m$ is continuous, it holds that $m\equiv0$.  Note that for almost any $t\geq0$, $\dot m(t)\in-\tau(\nabla^2\varphi(x(t)))^{-1}(m(t)-D_f(x(t)))$, thus $0\in(\nabla^2\varphi(x(t)))^{-1}D_f(x(t))$ holds for almost all $t\geq0$, which leads to a contradiction. Therefore, for both cases, there exists $T>0$ such that $h(x(T),m(T))<h(x(0),m(0))$. This completes the proof.  
\end{proof}

By Lemma \ref{le:asymp_gradient}, Proposition \ref{prop:MSBG_Lyapunov}, and Theorem \ref{convergence-thm-func-val}, we have the following convergence results for \eqref{Eq:MSBG_update}.
\begin{theorem}
Suppose Assumptions \ref{assumption_single} and \ref{assumption:Lyapunov_BGD} hold. {Then almost surely,} any cluster point of $\{x_k\}$ generated by \eqref{Eq:MSBG_update} is a $D_f$-stationary point of $f$, $\lim_{k\rightarrow\infty}m_k=0$, and the function values $\{f(x_k)\}$ converge.
\end{theorem}
\begin{proof}
Theorem \ref{convergence-thm-func-val} implies that any cluster point of $\{(x_k,m_k)\}$ lies in $\{(x,m):0\in D_f(x),m=0\}$, and $\{f(x_k)+\frac{1}{2\tau}\norm{m_k}^2\}$ converges. For any convergent subsequence $x_{k_j}\rightarrow\bar x$, since $\{m_k\}$ is bounded, then there exist subsequence $\{m_{k_{j_i}}\}$ such that $m_{k_{j_i}}\rightarrow\bar m$. Therefore, $(x_{k_{j_i}},m_{k_{j_i}})\rightarrow(\bar x,\bar m)$. Then, it holds that $0\in D_f(\bar x)$. Similarly, we can prove that for any convergent subsequence $\{m_{k_j}\}$ such that $m_{k_j}\rightarrow\bar m$, we have that $\bar m=0$. Therefore, $\lim_{k\rightarrow\infty}m_k=0$. Then, $\lim_{k\rightarrow\infty}f(x_k)=\lim_{k\rightarrow\infty}f(x_k)+\frac{1}{2\tau}\norm{m_k}^2$. This completes the proof.
\end{proof}

\subsection{Stochastic Bregman proximal subgradient method}
In this section, we consider solving the following constrained composite optimization problem:
\begin{equation}
    \label{Eq_constrained}
    \min_{x\in\cX}\;h(x):=f(x)+R(x),
\end{equation}
where $\mathcal{X}$ is a closed subset of $\mathbb{R}^n$, and $R$ is a locally Lipschitz function with an efficiently computable conservative field. In many applications, $R$ serves as the regularization function, which is usually Clarke regular, and $\partial R$ is efficient to compute.  We consider applying the follow Bregman proximal subgradient method to solve \eqref{Eq_constrained},
\begin{equation}
    \label{Eq:SBPG}
    \tag{SBPG}
    \left\{\begin{aligned}
    x_{k+1}&\approx\argmin_{x\in\cX}\left\{\inner{g_k,x-x_k}+\frac{1}{\eta_k}\cD_\phi(x,x_k)+R(x)\right\},\\
    \text{s.t. }&\inner{g_k,x_{k+1}-x_k}+\frac{1}{\eta_k}\cD_\phi(x_{k+1},x_k)+R(x_{k+1})\leq R(x_k),\text{ and}\\
    &\dist\left(0,g_k+\frac{1}{\eta_k}(\nabla\phi(x_{k+1})-\nabla\phi(x_k))+D_R(x_{k+1})+N^L_{\cX}(x_{k+1})\right)\leq\nu_k.
    \end{aligned}\right.
\end{equation}
where $g_k=d_{f,k}+\xi_{k}$, $d_{f,k}\in D^{\delta_k}_f(x_{k})$. We can reformulate \eqref{Eq:SBPG} in the form of \eqref{Eq:Breg_general_iterative} as follows:
\[
\begin{aligned}
\nabla\phi(x_{k+1})=\nabla\phi(x_k)-\eta_k(d_{f,k}+d_{R,k}+d_{\cX,k}+d_{e,k}+\xi_k),
\end{aligned}
\]
where $d_{f,k}\in D^{\delta_k}_f(x_k)$, $d_{R,k}\in D_R(x_{k+1})$, $d_{\cX,k}\in N^L_{\cX}(x_{k+1})$, and $\norm{d_{e,k}}\leq\nu_k$. When $\cX = \R^n$, with $\delta_k = 0$ and $\nu_k = 0$, it follows that $x_{k+1} = {\rm T}^{\phi}_{\eta_k,R}(x_k, g_k)$. Let $d_k:=d_{f,k}+d_{R,k}+d_{\cX,k}+d_{e,k}$. We can easily verify that there exists $\{\tilde\delta_k\}$ such that $\lim_{k\rightarrow\infty}\tilde\delta_k=0$, and $d_k\in\cH^{\tilde\delta_k}(x_k)$, where $\cH:=D_f+D_R+N_{\cX}$. This leads to a differential inclusion for the proximal updates given by
\begin{equation}
    \frac{d}{dt}\nabla\phi(x(t))\in-\cH(x(t)),\;\text{where }\cH=D_f+D_R+N_{\cX}.
    \label{Eq:DI_prox}
\end{equation}

The momentum technique can also be integrated into \eqref{Eq:SBPG}, as illustrated in Section 4.2. For the sake of readability, we omit this extension. Note that neither $N_\cX$ nor $N^L_{\cX}$ is locally bounded, thus the results that rely on local boundedness assumption such as those presented in \cite{xiao2023convergence,ding2023adam} cannot be directly applied. We make the following assumptions on \eqref{Eq:SBPG}.
\begin{assumption}
    \label{assumption_prox_extension}
    \begin{enumerate}
        \item $\phi$ is a supercoecive Legendre kernel function over $\R^n$. Moreover, $\phi$ is locally strongly convex and $\nabla\phi$ is locally Lipschitz continuous. Additionally, for any absolutely continuous mapping $z(\cdot)\in\cC(\R_+,\R^n)$, $\nabla^2\phi(z(s))$ is positive definite for almost all $s\geq0$. 
        \item The sequences $\{x_k\}$, $\{\nabla\phi(x_k)\}$, $\{d_{f,k}\}$ and $\{d_{R,k}\}$ are uniformly bounded almost surely.
        \item $\{\xi_k\}$ is a uniformly bounded martingale difference noise and $\sup_{k\geq0}\norm{\xi_k}<\infty$. The stepsize sequence $\{\eta_k\}$ satisfies  $\sum_{k=0}^\infty\eta_k=\infty$ and $\eta_k=o\left(\frac{1}{\log k}\right)$. 
        \item $\lim_{k\rightarrow\infty}\delta_k=0$, and $\lim_{k\rightarrow\infty}\nu_k=0$.
        \item For any $\eta>0$, $\eta R+\phi$ is supercoecive.
    \end{enumerate}
\end{assumption}
{
Now, we make some remarks on Assumption \ref{assumption_prox_extension}. Assumption \ref{assumption_prox_extension}(1) and (3) are directly adapted from Assumption \ref{Assumption: DI}(1) and (4), respectively. Assumption \ref{assumption:SBGD}(4) characterizes how $d_k$ approximates $D_f(x_k)$ and quantifies the inexactness in solving the subproblem. Assumption \ref{assumption:SBGD}(2) can be ensured by the global stability when the noise sequence $\{\xi_k\}$ corresponds to randomly reshuffling noises, as demonstrated by Theorem \ref{The_convergence_Nondiminishing_RR}. Assumption \ref{assumption_prox_extension}(5) ensures the well-posedness of Bregman proximal mapping as demonstrated in Lemma \ref{well-def-prox-mapping}.}

To ensure the convergence of \eqref{Eq:SBPG}, we make the following assumptions on $f$ and kernel $\phi$, which is essentially a specific case of Assumption \ref{assumption_Sard_Lyapunov}.

\begin{assumption}
    \label{assumption_Sard_prox}
    \begin{enumerate}
    \item $h$ is lower bounded. Moreover, the locally Lipschitz continuous functions $f$ and $R$ are potential functions that admit convex-valued $D_f$ and $D_R$ as their conservative fields, respectively. $\cX$ admits the chain rule as described in Definition \ref{def:chain_rule_set}.
    \item The critical value set $\{h(x):0\in D_f(x)+D_R(x)+N_{\cX}(x)\}$ has empty interior in $\R$. 
    \end{enumerate}
\end{assumption}

\begin{lemma}
    \label{le:SBPG_approx_map}
    Suppose Assumptions \ref{assumption_prox_extension} and \ref{assumption_Sard_prox} hold. Let $d_k:=d_{f,k}+d_{R,k}+d_{\cX,k}+d_{e,k}$. For any increasing sequence $\{k_j\}$ such that $\{x_{k_j}\}$ converges to $\bar x$, it holds that 
    \[
    \lim_{N\rightarrow\infty}\dist\left(\frac{1}{N}\sum_{j=1}^{N}d_{k_j},\cH(\bar x)\right)=0,
    \]
    where $\cH:=D_f+D_R+N_{\cX}$. 
\end{lemma}
\begin{proof}
Given the condition 
\[
R(x_{k+1})+\inner{d_{f,k}+\xi_k,x_{k+1}-x_k}+\frac{1}{\eta_k}\cD_\phi(x_{k+1},x_k)\leq R(x_k),
\]
we derive that 
\[
\frac{\cD_\phi(x_{k+1},x_k)}{\norm{x_{k+1}-x_k}}\leq\eta_k\frac{|R(x_{k+1})-R(x_k)|}{\norm{x_{k+1}-x_k}}+\eta_k(\norm{d_{f,k}+\xi_k}).
\]
Assumption \ref{assumption_prox_extension} ensures that $\sup_{k\geq0}\frac{1}{\eta_k}\norm{\nabla\phi(x_{k+1})-\nabla\phi(x_k)}<\infty$. Moreover, we have that for some $\norm{d_{e,k}}\leq\nu_k$, 
\[
\frac{\nabla\phi(x_{k+1})-\nabla\phi(x_k)}{\eta_k}=-\left(d_{f,k}+d_{R,k}+d_{\cX,k}+d_{e,k}+\xi_k\right).
\]
Note that the left hand side is uniformly bounded, $\{d_{f,k}\}$, $\{d_{R,k}\}$, $\{d_{e,k}\}$ and $\{\xi_k\}$ are all uniformly bounded, therefore, it holds that $\{d_{\cX,k}\}$ is also uniformly bounded. For any $\{x_{k_j}\}$ converging to $\bar x$, by the outer semicontinuity of $D_f$, $D_R$ and $N_{\cX}^L$, along with \cite[Theorem 5.7, Proposition 5.12]{RockWets98}, it holds that 
\[
\lim_{j\rightarrow\infty}\dist\left(d_{x,k_j},D_f(\bar x)\right)=0,\;\lim_{j\rightarrow\infty}\dist\left(d_{R,k_j},D_R(\bar x)\right)=0,\;\text{and }\lim_{j\rightarrow\infty}\dist\left(d_{\cX,k_j},N^L_{\cX}(\bar x)\right)=0.
\]
Note that $N^L_{\cX}\subset N_{\cX}$,
therefore, $\lim_{j\rightarrow\infty}\dist\left(d_{k_j},\mathcal{H}(\bar x)\right)=0$. By Jensen's inequality, we prove this lemma. 
\end{proof}

\begin{proposition}
    \label{prop:SBPG_Lyapunov}
     Suppose Assumptions \ref{assumption_prox_extension} and \ref{assumption_Sard_prox} hold. Then the function $h$ in \eqref{Eq_constrained} is a Lyapunov function for the differential inclusion \eqref{Eq:DI_prox} with stable set $\{x\in\R^n:0\in D_f(x)+D_R(x)+N_{\cX}(x)\}$.    
\end{proposition}
\begin{proof}
Consider any trajectory $z(t)$ of \eqref{Eq:DI_prox} with $0\notin\cH(z(0))$. By the chain rule, it holds that for almost all $t\geq0$,
\[
f(z(t))'=\inner{D_f(z(t)),\dot z(t)},\;R(z(t))'=\inner{D_R(z(t)),\dot z(t)},\;0=\inner{N_{\cX}(z(t)),\dot z(t)}.
\]
Then, we have that $h(z(t))'=\inner{\cH(z(t)),\dot z(t)}$ for almost all $t\geq0$. Note that for almost all $t\geq0$, it holds that $\nabla^2\phi(z(t))\dot z(t)\in-\cH(z(t))$. Then, we have 
\[
\begin{aligned}
&h(z(t))-h(z(0))=\int_{0}^t\inner{\cH(z(s)),\dot z(s)}ds
=-\int_0^t\inner{\nabla^2\phi(z(s))\dot z(s),\dot z(s)}ds\\
\leq&-\int_0^t\lambda_{\min}(\nabla^2\phi(z(s)))\norm{\dot z(s)}^2ds.
\end{aligned}
\]
If there exists no $t>0$, such that $h(z(t))<h(z(0))$. Then,  we have that $\dot z(t)=0$ for almost all $t\geq0$. Thus, $z(t)\equiv z(0)$. Therefore, we have that $0=\frac{d}{dt}\nabla\phi(z(s))\in-\cH(z(0))$, which is contradictory to the fact that $0\notin\cH(z(0))$. This completes the proof.
\end{proof}

By Lemma \ref{le:SBPG_approx_map}, Proposition \ref{prop:SBPG_Lyapunov} and Theorem \ref{convergence-thm-func-val}, we can directly derive the following convergence results.
\begin{theorem}
    \label{thm:SBPG}
    Suppose Assumptions \ref{assumption_prox_extension} and \ref{assumption_Sard_prox} hold. {Then almost surely,} any cluster point of $\{x_k\}$ generated by \eqref{Eq:SBPG} is a $\cH$-stationary point and the function values $\{f(x_k)+R(x_k)\}$ converge.
\end{theorem}

\section{Numerical experiments}
In this section, we conduct preliminary numerical experiments to illustrate the performance of our proposed methods, focusing on training nonsmooth neural networks for image classification and language modeling tasks. These experiments are performed using an NVIDIA RTX 3090 GPU, and implemented in Python 3.9 with PyTorch version 1.12.0.

Our experiments employ a polynomial kernel-based stochastic Bregman subgradient method to train nonsmooth neural networks. Specifically, we use a blockwise polynomial kernel function $\phi(x) = \sum_{i=1}^Lp_i(\|x_i\|)$, where $x=(x_1,...,x_L)$ represents the concatenation of all layers' parameters in a neural network with $L$ layers, and each $p_i$ is a univariate polynomial of degree at least 2. When $p_i(\lambda) = \frac{1}{2}\lambda^2$, this approach becomes equivalent to SGD. The polynomial $p_i(\lambda) = \frac{1}{2}\lambda^2 + \frac{\sigma}{r}\lambda^r$ with $r\geq 4$, as discussed in the prior work \cite{ding2023nonconvex}, is applied in our numerical experiments. In this case, the update scheme is defined as follows
\begin{equation}
    \label{alg:MSBPG}
    \tag{MSBG}
    \begin{aligned}
    x_{k+1}&=\nabla\phi^*(\nabla\phi(x_k)-\eta_k m_k)\\
    m_{k+1}&=m_k-\theta_k(\nabla^2\phi(x_k))^{-1}(m_k-g_k),
    \end{aligned}
\end{equation}
where the calculation of $x_{k+1}$ involves solving a nonlinear equation. Given the allowance of inexact solutions for the subproblems of MSBG, as mentioned earlier, we consider an alternative preconditioned update scheme,
\begin{equation}
    \label{iMSBG}
    \tag{iMSBG}
    \begin{aligned}
     x_{k+1}&=x_k-\eta_k(\nabla^2\phi(x_k))^{-1}m_k\\
     m_{k+1}&=m_k-\theta_k(\nabla^2\phi(x_k))^{-1}(m_k-g_k),
    \end{aligned}
\end{equation}
which avoids solving a nonlinear equation in the first step of \eqref{alg:MSBPG}. By the Sherman-Morrison formula, we have
\[
(\nabla|_{x_i}^2p_i(\|x_i\|))^{-1}=\frac{1}{\sigma\norm{x_i}^{r-2}}I-\frac{\sigma(r-2)\norm{x_i}^{r-4}}{(1+\sigma\norm{x_i}^{r-2})^2+\sigma(r-2)(1+\sigma\norm{x_i}^{r-2})\norm{x_i}^2}x_ix_i^T,
\]
and $\nabla^2\phi(x)=\text{diag}\left((\nabla|_{x_1}^2p_1(\|x_1\|))^{-1},...,(\nabla|_{x_L}^2p_L(\|x_L\|))^{-1}\right)$ is block diagonal. We employ this kernel function and use the notation MSBG$K$/iMSBG$K$ to denote MSBG/iMSBG with the polynomial degree parameter $r$ set to $K$. Our experiments focus on two main applications: training Convolutional Neural Networks (CNNs) for image classification and Long Short-Term Memory (LSTM) \cite{hochreiter1997long} networks for language modeling. Specifically, our image classification experiments include training  Resnet14 and ResNet34 \cite{he2016deep} on CIFAR-10 and CIFAR-100 datasets \cite{krizhevsky2009learning}. Our language modeling experiments focus on 1-layer, 2-layer, and 3-layer LSTM networks applied to the Penn Treebank dataset \cite{marcus1993building}. 

\paragraph{CNNs on image classification} 
For the CNN experiments, we set the stepsize $\eta_s$ for each epoch $s$ as $\eta_s = \frac{\eta_0}{1 + (\log(s+1))^{1.1}}$, where $\eta_0$ is the initial stepsize. The momentum parameters are all set to $\theta_s=\frac{0.1}{1+\log(s+1)^{1.1}}$. For MSBG4 and iMSBG4, we choose $\sigma=0.01$, and for iMSBG6, $\sigma=0.0001$. We search the initial stepsize $\eta_0$ among the grid $\{0.001,0.01,0.1,1.0\}$ and select the value that achieves the highest test accuracy. The results are shown in Figure \ref{fig:cifar10} and Figure \ref{fig:cifar100}. We can observe that by selecting a proper kernel function, our Bregman subgradient methods can outperform SGD in terms of test accuracy. Moreover, we can see that MSBG4 and iMSBG4 have similar performance, although iMSBG4 solves the subproblem inexactly.
 
 Additionally, we evaluate the robustness of the selection of initial stepsize $\eta_0$, as demonstrated in Figure \ref{fig:robust}. We can see that the peak test accuracies of all methods are similar, yet our Bregman subgradient methods demonstrate a wider effective initial step size range, indicating a reduced sensitivity to the choice of initial step size -- a benefit attributable to the kernel function.

\begin{figure*}[th]
		\begin{center}
			\setlength{\tabcolsep}{0.0pt}  
			\scalebox{1}{\begin{tabular}{ccc}
					\includegraphics[width=0.33\linewidth]{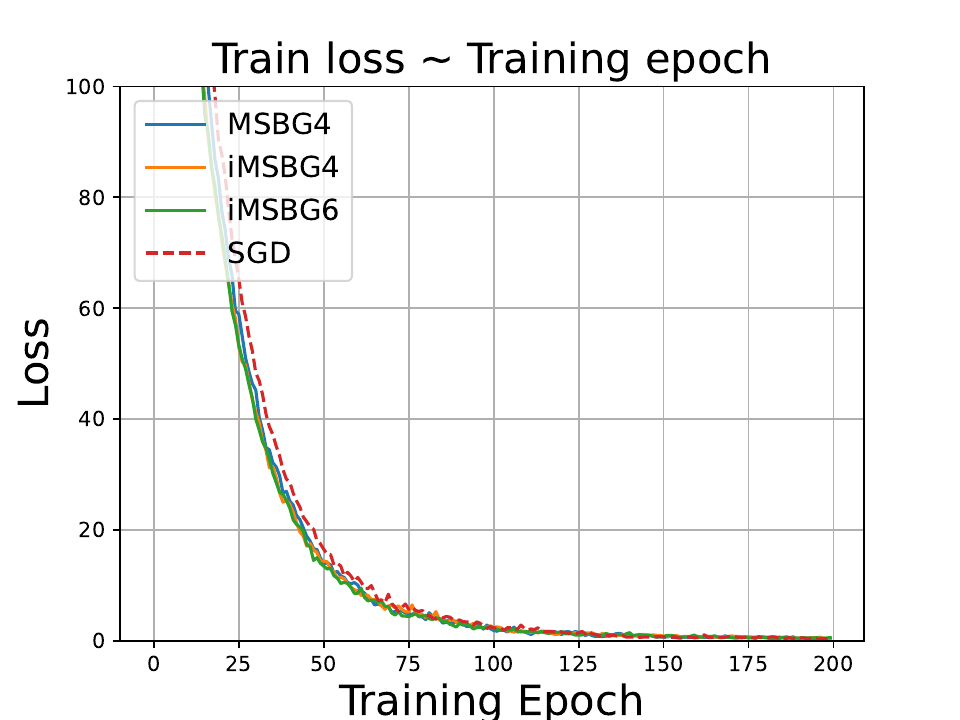}&
					\includegraphics[width=0.33\linewidth]{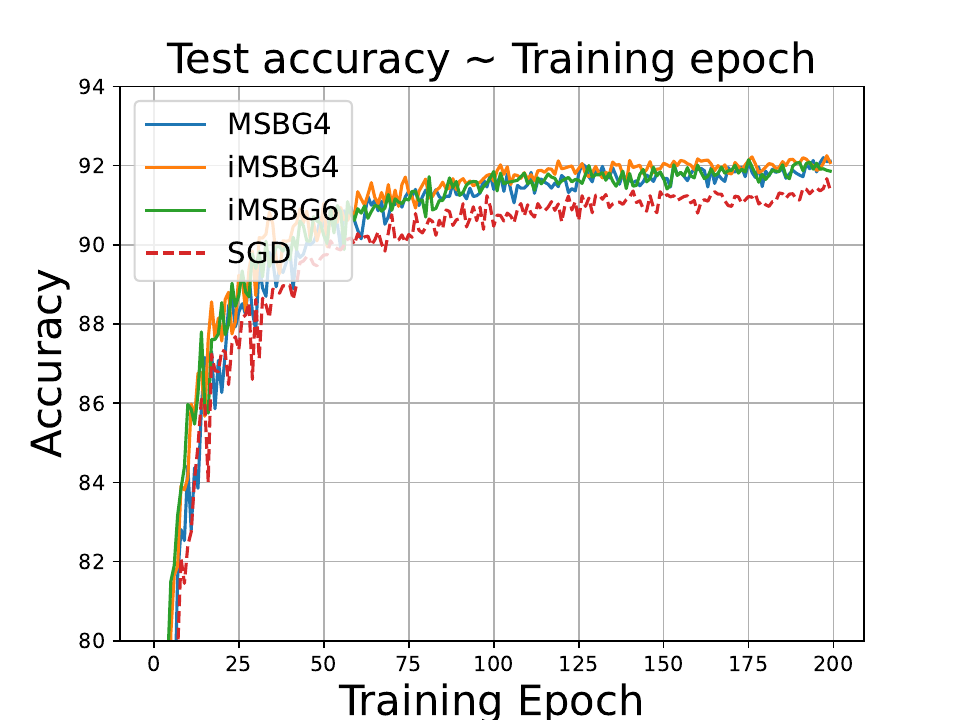}&
					\includegraphics[width=0.33\linewidth]{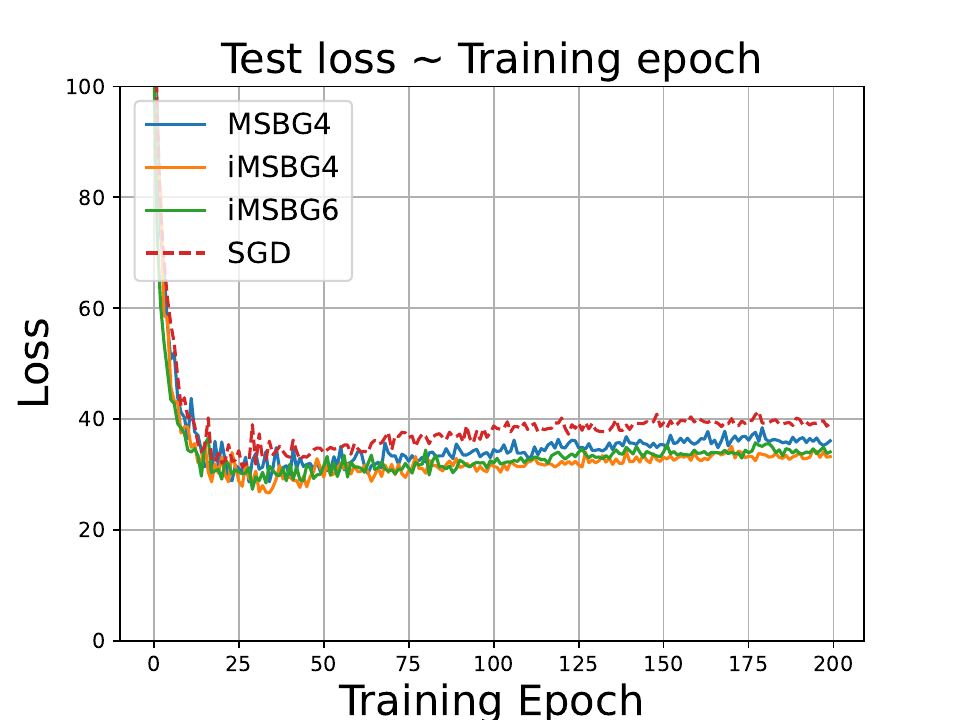}\\
		
					\footnotesize{(a) Train loss} &  \footnotesize{(b) Test accuracy}
					&  \footnotesize{(c) Test loss}
					\\
			\end{tabular}}
		\end{center}
		\caption{Resnet 14 on CIFAR10.} \label{fig:cifar10}
\end{figure*}

\begin{figure*}[th]
		\begin{center}
			\setlength{\tabcolsep}{0.0pt}  
			\scalebox{1}{\begin{tabular}{ccc}
					\includegraphics[width=0.33\linewidth]{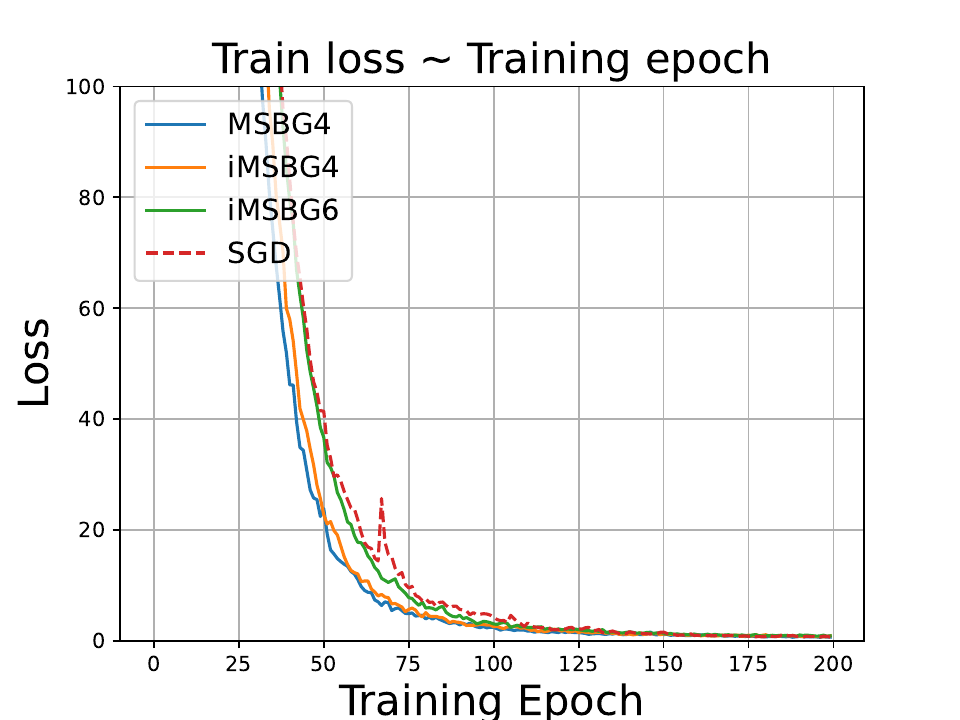}&
					\includegraphics[width=0.33\linewidth]{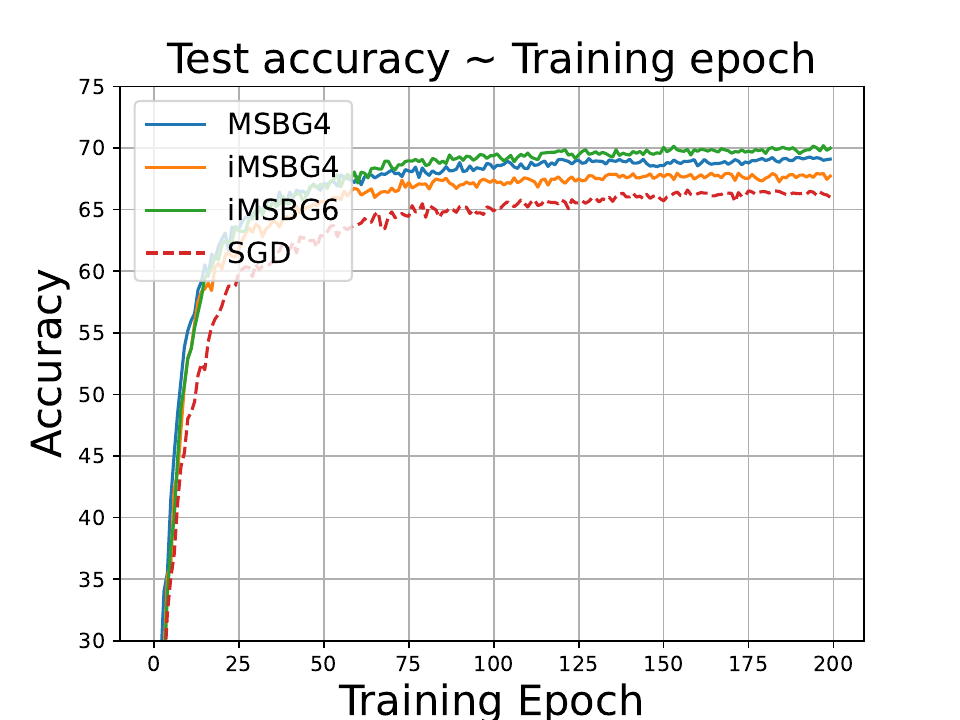}&
					\includegraphics[width=0.33\linewidth]{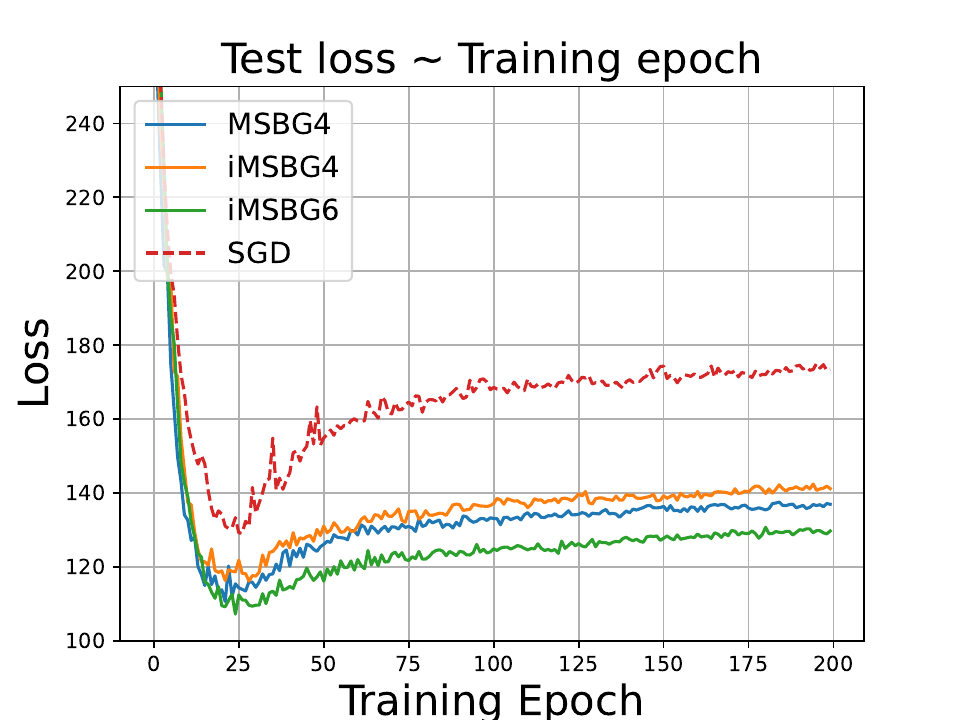}\\
		
					\footnotesize{(a) Train loss} &  \footnotesize{(b) Test accuracy}
					&  \footnotesize{(c) Test loss}
					\\
			\end{tabular}}
		\end{center}
		\caption{Resnet 34 on CIFAR100.} \label{fig:cifar100}
\end{figure*}

\begin{figure*}[th]
		\begin{center}
			\setlength{\tabcolsep}{0.0pt}  
			\scalebox{1}{\begin{tabular}{cccc}
	
					\includegraphics[width=0.33\linewidth]{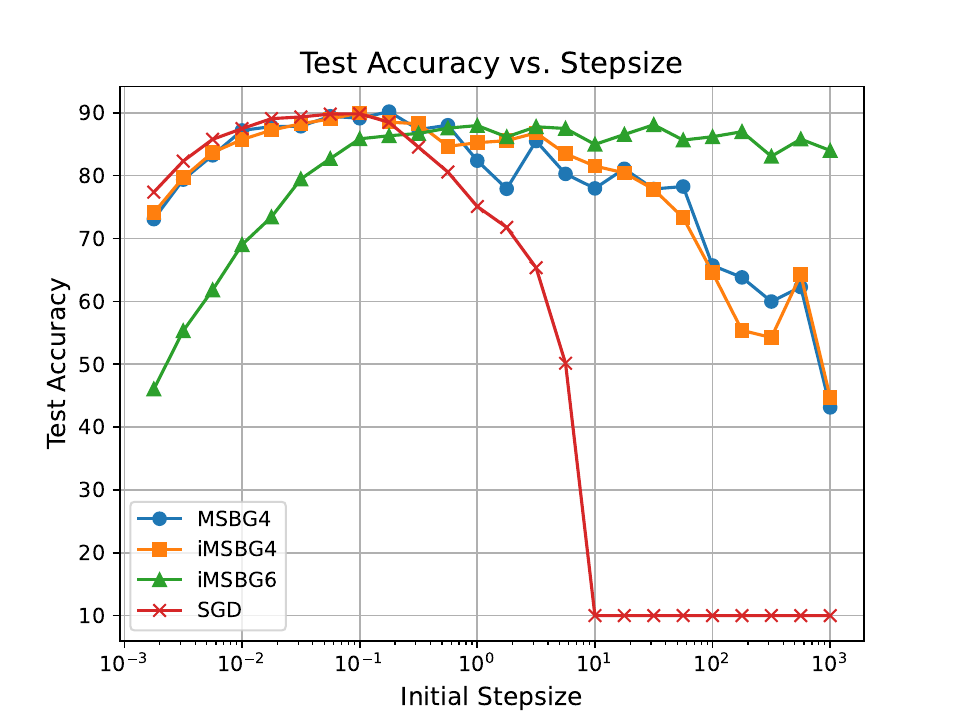}&
					\includegraphics[width=0.33\linewidth]{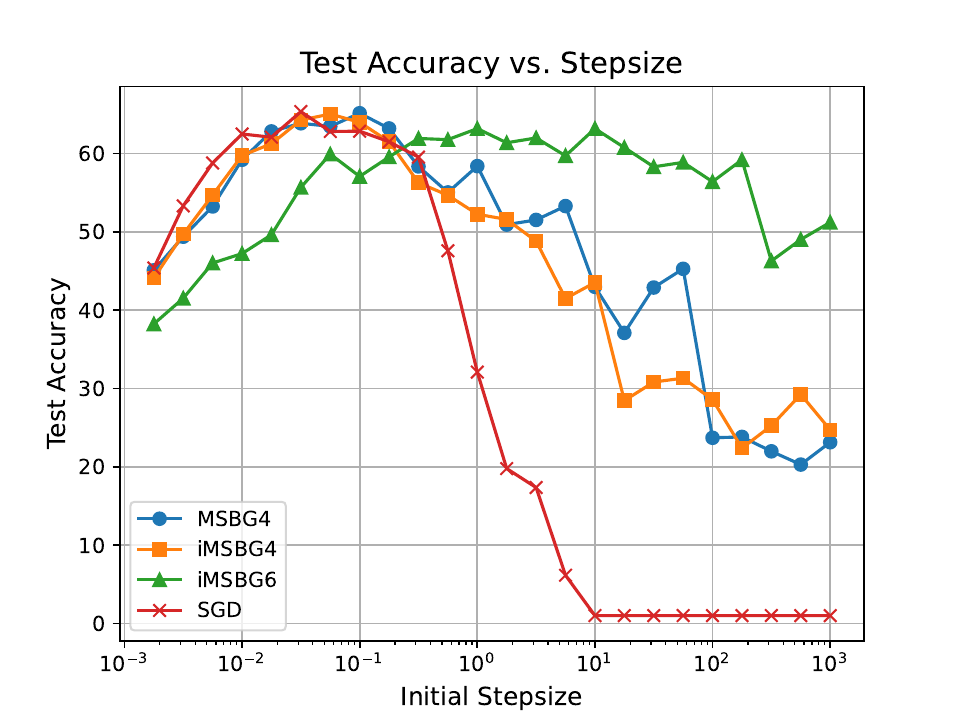} &
					\\
					 \multicolumn{1}{c}{\footnotesize{(a) ResNet14 on CIFAR10}}&
					\multicolumn{1}{c}{\footnotesize{(b) ResNet18 on CIFAR100}}                  
			\end{tabular}}
		\end{center}
		\caption{Robustness for initial stepsize. Figure (a) reports the test accuracy for 40 epochs. Figure (b) reports the test accuracy for 30 epochs.} \label{fig:robust}
	\end{figure*}

\paragraph{LSTMs on language modeling}
For the LSTM experiments, we initially set the stepsize as a constant. The stepsize is then decreased to 0.1 times its previous value at both the 150th and 300th epochs. After 300 epochs, we set $\eta_s=\frac{0.01\eta_0}{1+\log(s-300)^{1.1}}$, with $s$ representing the epoch number. Here $\eta_0$ is the initial stepsize. Within the s-th epoch, $\eta_k$ takes the constant value $\eta_s$. Similarly, the momentum parameters are all set to $\theta_s=\frac{0.1}{1+\log(s+1)^{1.1}}$. For MSBG4 and iMSBG4, we set $\sigma=10^{-6}$. We search $\eta_0$ among the grid $\{1,10,20,40,80,100\}$ and report the results based on achieving the highest test accuracy. The results are shown in Figures \ref{fig:LSTM1}, \ref{fig:LSTM2}, and \ref{fig:LSTM3}. We can observe that selecting an appropriate kernel function enables our Bregman subgradient methods to achieve superior test accuracy compared to SGD. 
 
 We also compare the one-epoch runtime for all considered methods over all tasks. We can observe in Table \ref{tab:computation_time} that the proposed inexact Bregman subgradient methods are nearly as efficient as SGD, largely because iMSBG circumvents the need to solve nonlinear equations in computing the Bregman proximal mapping.

\begin{figure*}[th]
		\begin{center}
			\setlength{\tabcolsep}{0.0pt}  
			\scalebox{1}{\begin{tabular}{cccc}
					\includegraphics[width=0.25\linewidth]{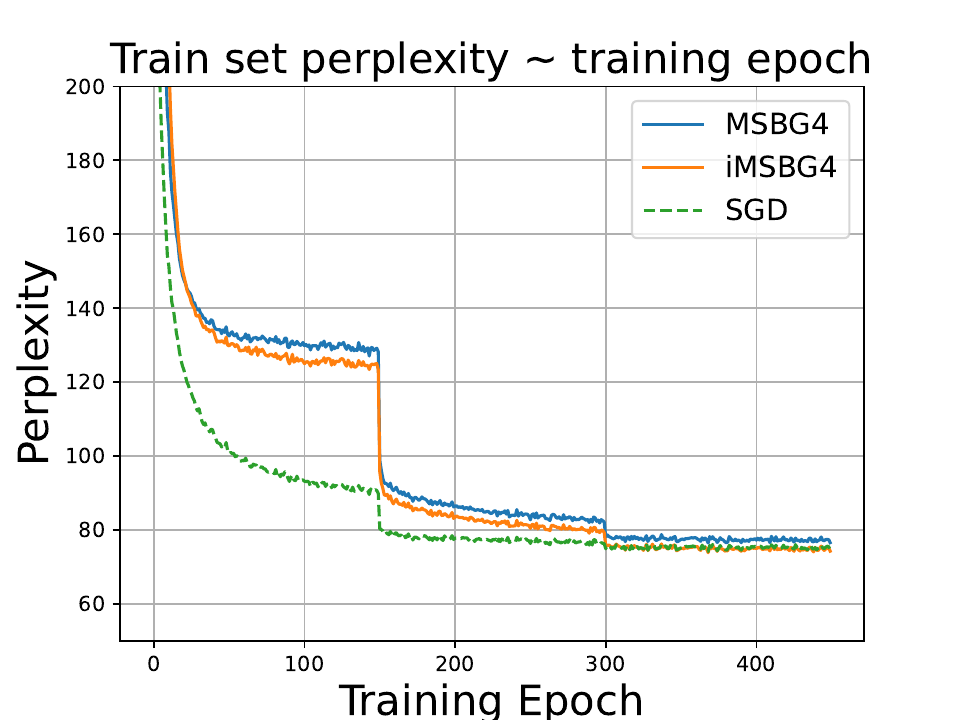}&
					\includegraphics[width=0.25\linewidth]{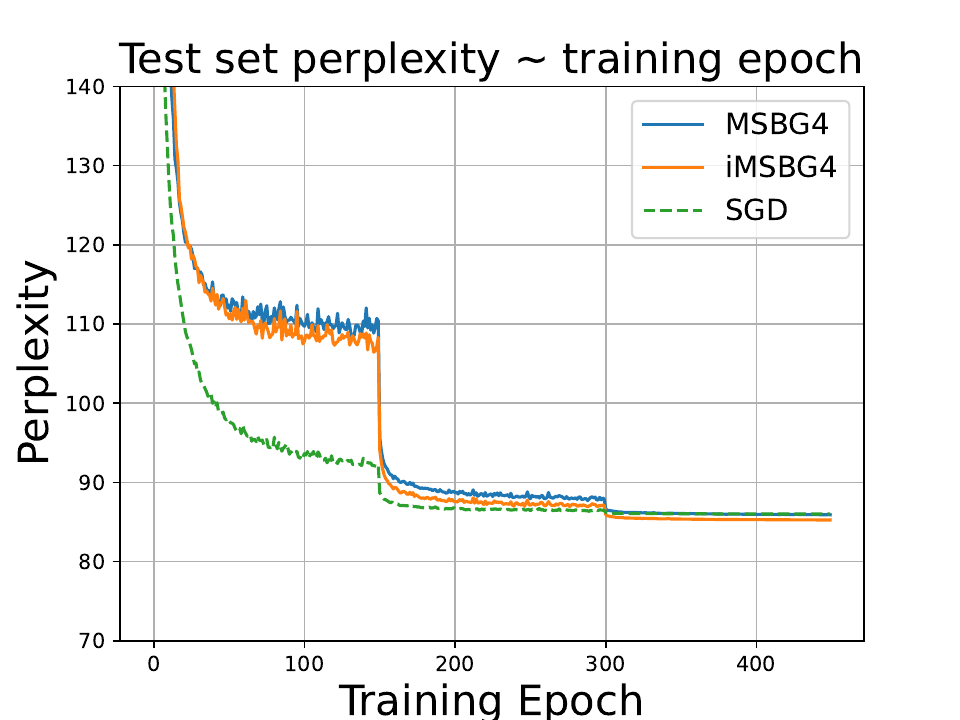}&
					\includegraphics[width=0.25\linewidth]{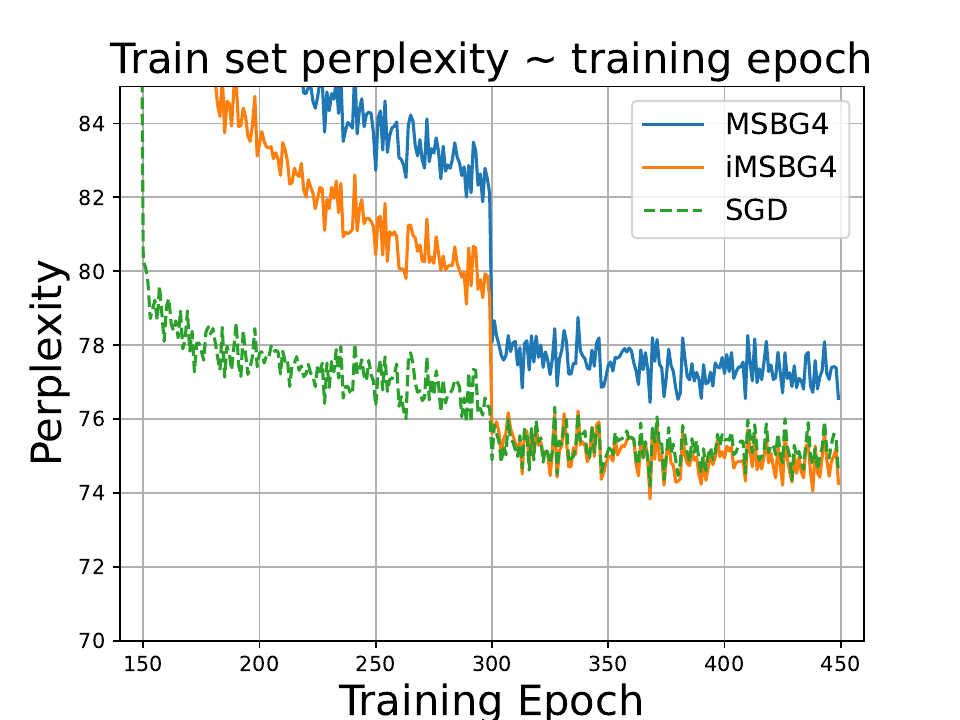} &
					\includegraphics[width=0.25\linewidth]{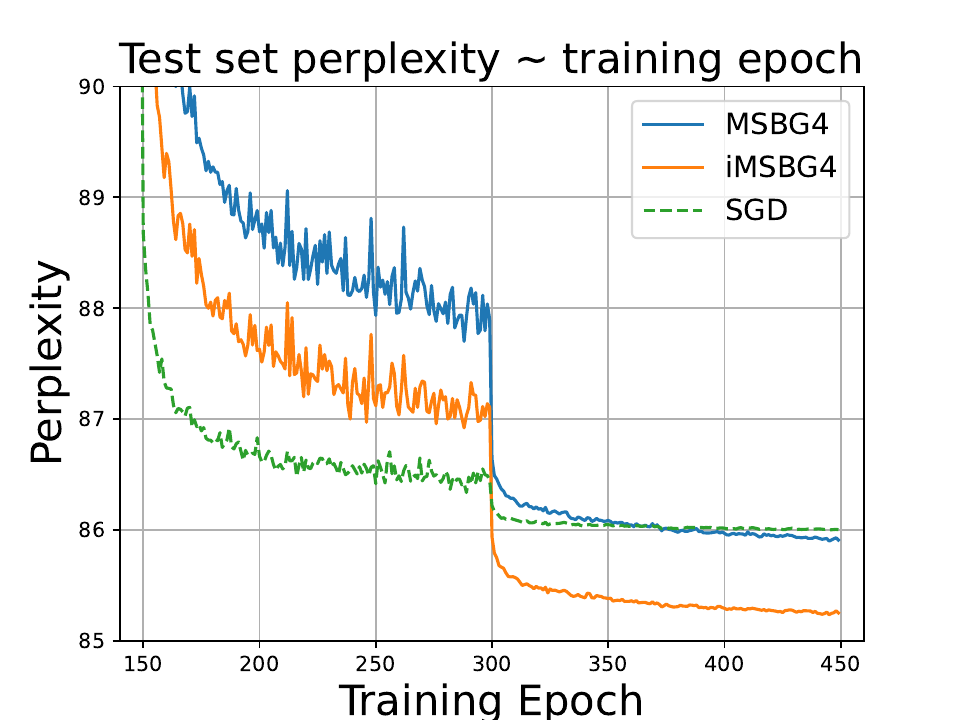}\\
					\multicolumn{1}{c}{\footnotesize{(a) Train perplexity}} &  \multicolumn{1}{c}{\footnotesize{(b) Test perplexity}}&
					\multicolumn{1}{c}{\footnotesize{(c) Train perplexity zooms in}}&
					\multicolumn{1}{c}{\footnotesize{(d) Test perplexity zooms in}}                  
			\end{tabular}}
		\end{center}
		\caption{1-layer LSTM on Penn Treebank dataset.} \label{fig:LSTM1}
	\end{figure*}
 \begin{figure*}[th]
		\begin{center}
			\setlength{\tabcolsep}{0.0pt}  
			\scalebox{1}{\begin{tabular}{cccc}
					\includegraphics[width=0.25\linewidth]{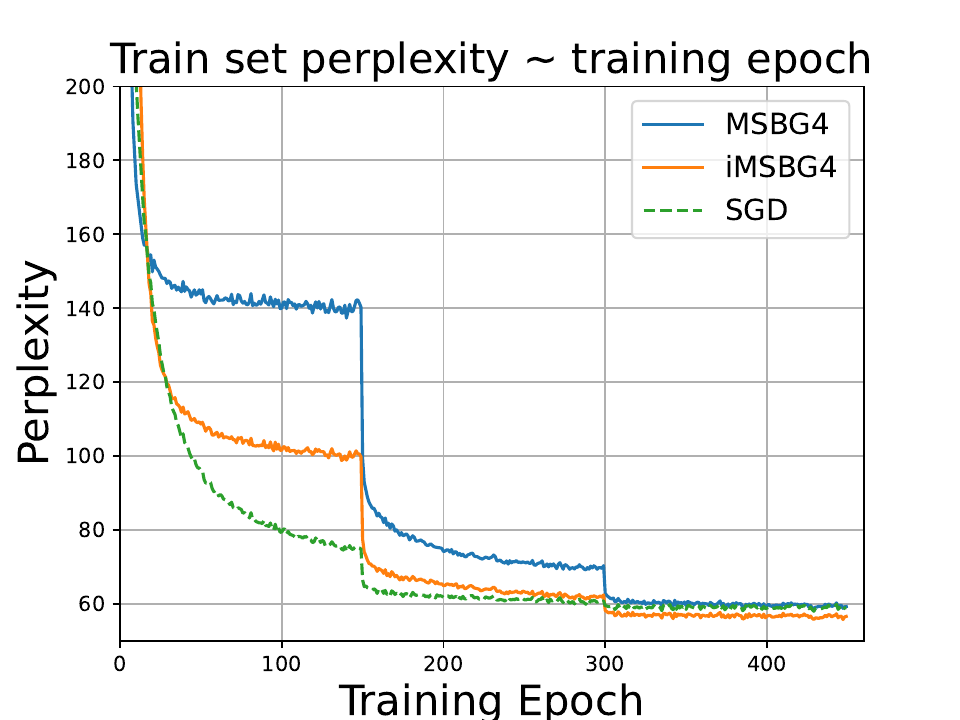}&
					\includegraphics[width=0.25\linewidth]{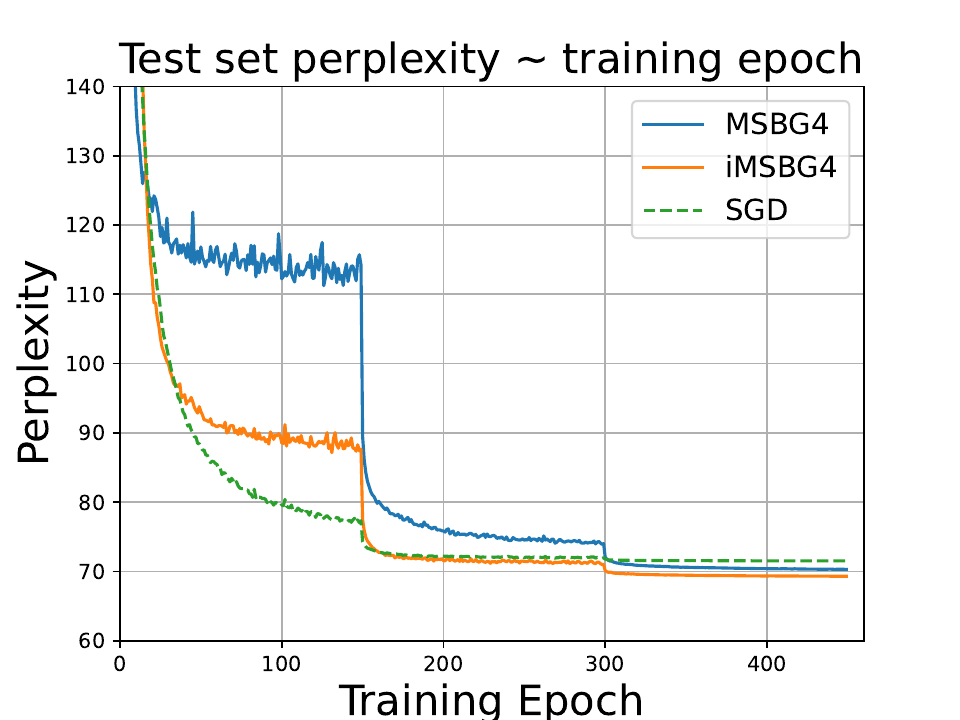}&
					\includegraphics[width=0.25\linewidth]{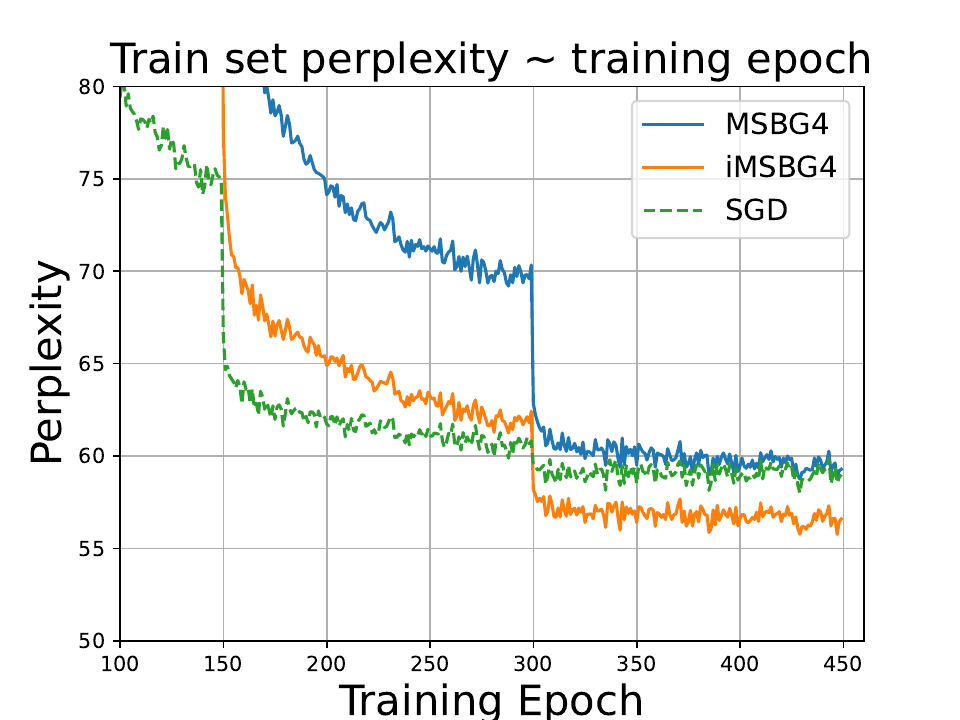} &
					\includegraphics[width=0.25\linewidth]{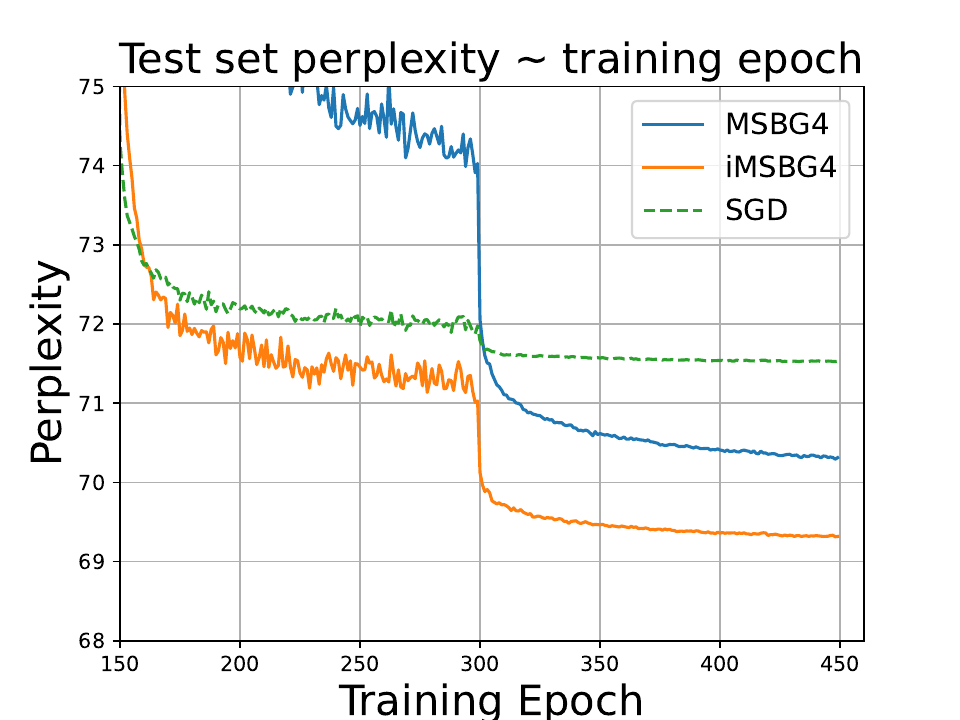}\\
					\multicolumn{1}{c}{\footnotesize{(a) Train perplexity}} &  \multicolumn{1}{c}{\footnotesize{(b) Test perplexity}}&
					\multicolumn{1}{c}{\footnotesize{(c) Train perplexity zooms in}}&
					\multicolumn{1}{c}{\footnotesize{(d) Test perplexity zooms in}}                  
			\end{tabular}}
		\end{center}
		\caption{2-layer LSTM on Penn Treebank dataset.} \label{fig:LSTM2}
	\end{figure*}

 \begin{figure*}[th]
		\begin{center}
			\setlength{\tabcolsep}{0.0pt}  
			\scalebox{1}{\begin{tabular}{cccc}
					\includegraphics[width=0.25\linewidth]{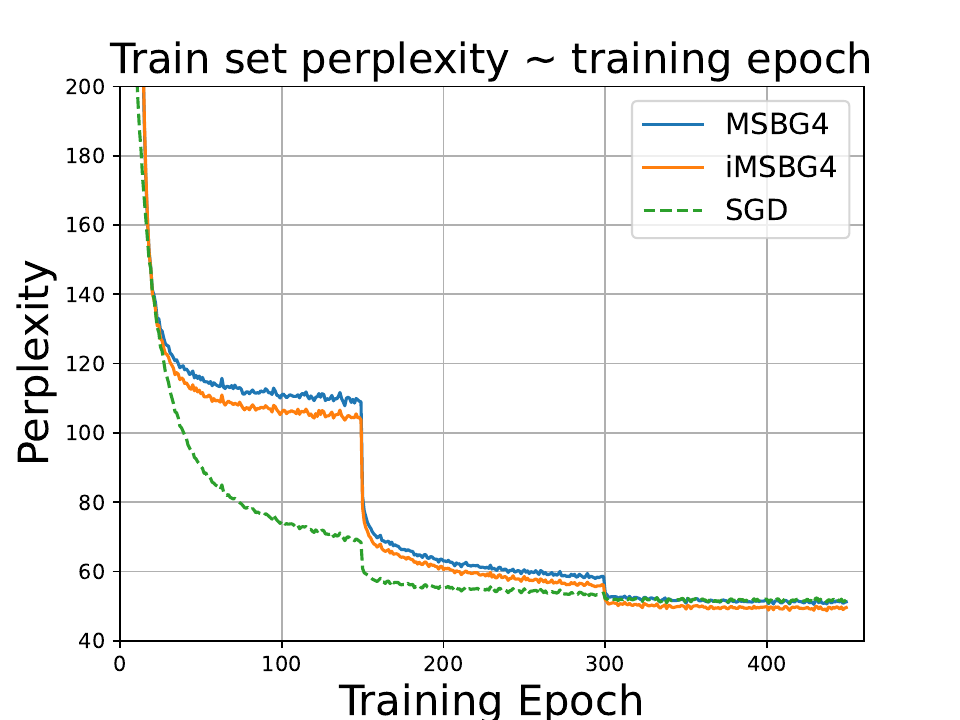}&
					\includegraphics[width=0.25\linewidth]{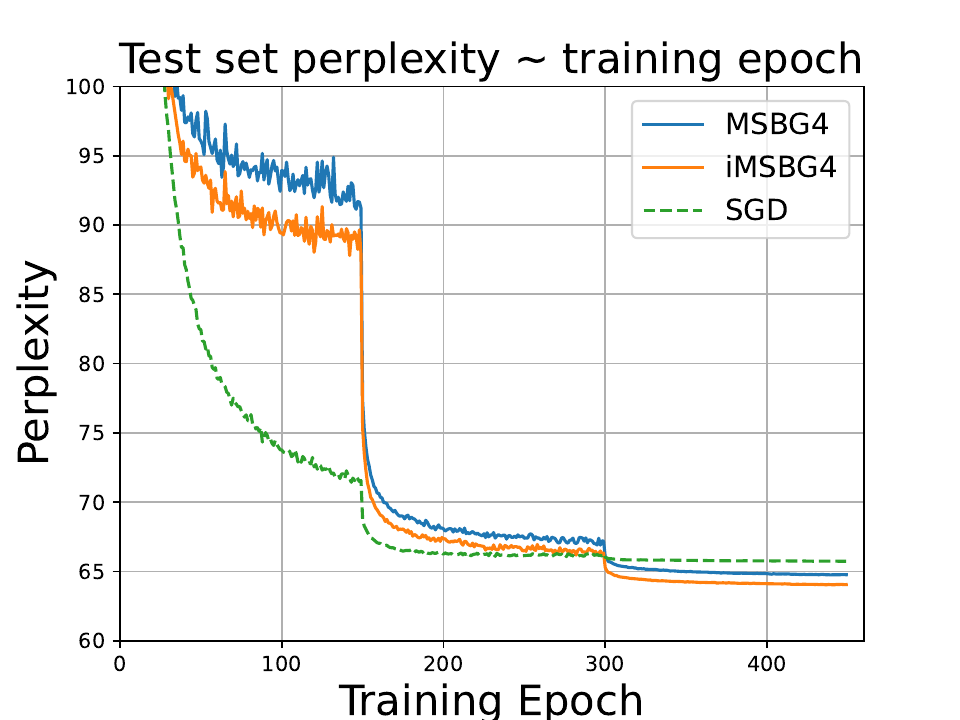}&
					\includegraphics[width=0.25\linewidth]{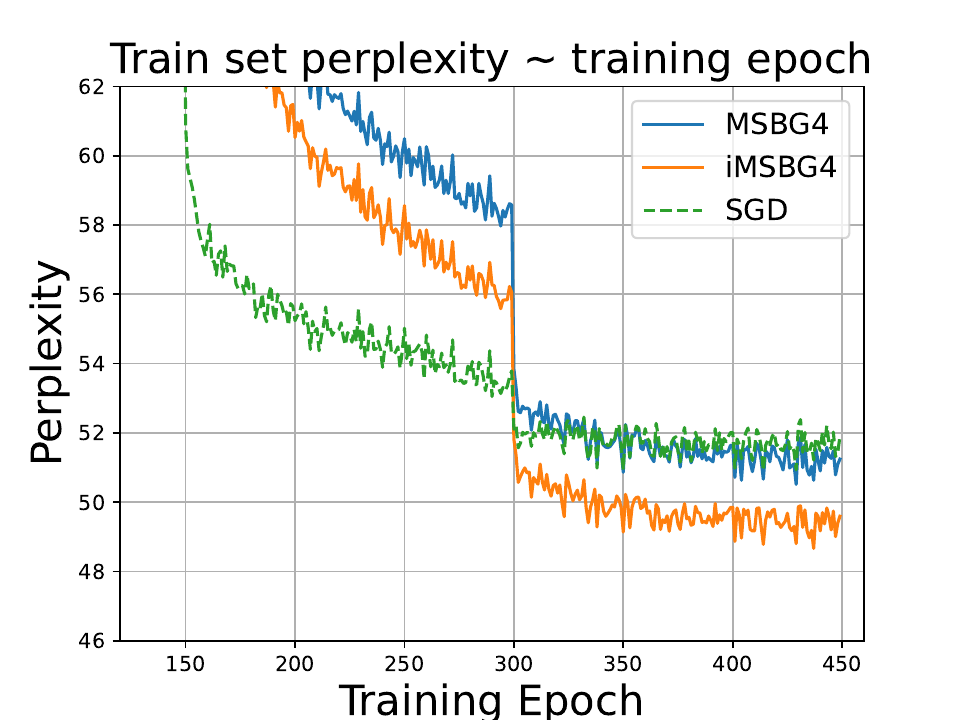} &
					\includegraphics[width=0.25\linewidth]{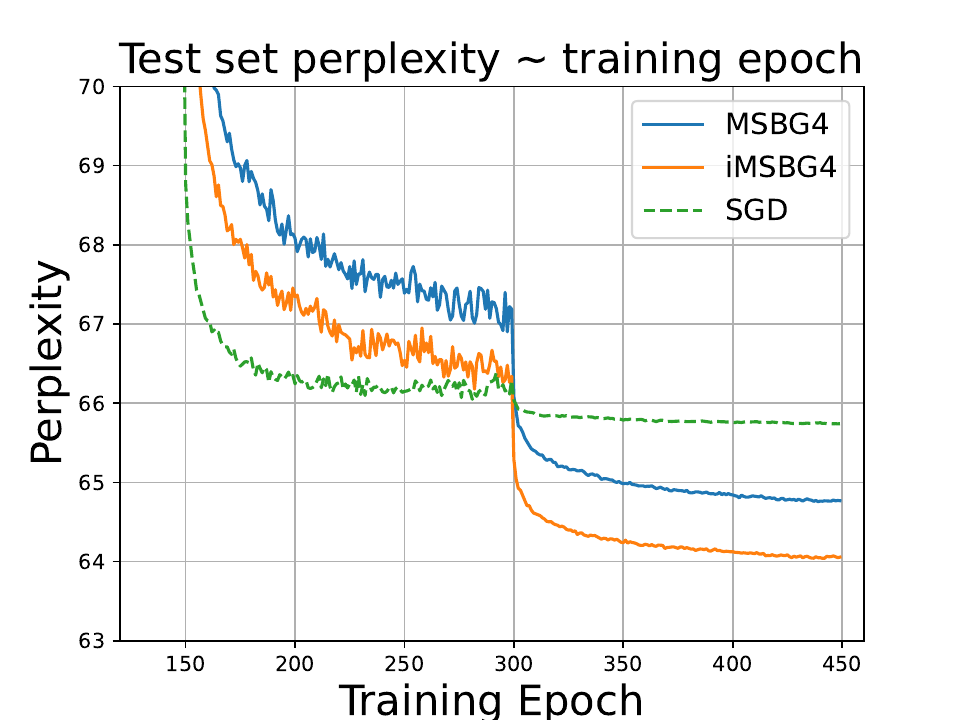}\\
					\multicolumn{1}{c}{\footnotesize{(a) Train perplexity}} &  \multicolumn{1}{c}{\footnotesize{(b) Test perplexity}}&
					\multicolumn{1}{c}{\footnotesize{(c) Train perplexity zooms in}}&
					\multicolumn{1}{c}{\footnotesize{(d) Test perplexity zooms in}}                  
			\end{tabular}}
		\end{center}
		\caption{3-layer LSTM on Penn Treebank dataset.} \label{fig:LSTM3}
	\end{figure*}

\begin{table}[ht]
\centering
\begin{tabular}{lcccccc}
\hline
\textbf{Method} & \textbf{Task 1} & \textbf{Task 2} & \textbf{Task 3} & \textbf{Task 4} & \textbf{Task 5} \\ \hline
SGD             & $18.11\pm0.32$              & $36.28\pm0.07$               & $5.56\pm0.03$              & $17.24\pm0.03$               & $30.11\pm0.06$              \\
MSBG4           & $20.57\pm 0.12$               & $40.43 \pm0.16$              & $6.48\pm 0.11$               & $ 18.97 \pm 0.12$              & $32.81\pm 0.23$              \\
iMSBG4          & $18.30\pm0.04$               & $37.25\pm0.07$               & $5.62\pm0.01$              & $18.45 \pm 0.04$               & $31.68\pm 0.27$              \\
iMSBG6          & $18.38\pm0.06$               & $37.24\pm0.12$               & $5.66\pm0.04$               & $18.52\pm0.04$              & $31.45\pm0.31$               \\ \hline
\end{tabular}
\caption{Computation Time of Each Epoch (in seconds). Task 1 is training ResNet14 on CIFAR10. Task 2 is training ResNet18 on CIFAR100. Task 3,4,5 are training one-layer, two-layer, three-layer LSTM on n Penn Treebank dataset, respectively.}
\label{tab:computation_time}
\end{table}

\section{Conclusions}

This paper explores Bregman subgradient methods for solving nonsmooth nonconvex optimization problems, particularly focusing on path-differentiable functions. We introduce a comprehensive stochastic Bregman framework that accommodates inexact evaluations of the abstract set-valued mapping. Employing a differential inclusion strategy and linear interpolation of dual sequences, we establish convergence results for our stochastic Bregman-type methods. This ensures that the discrete sequence subsequentially converges to the stable set of the differential inclusion, as well as the convergence of the Lyapunov function values. For applications, we demonstrate that stochastic Bregman subgradient methods, even with subproblems being solved inexactly, fit within our general framework, and we establish their convergence properties. Moreover, we integrate a momentum technique into the stochastic Bregman subgradient methods. Additionally, we extend our methodology to a proximal variant of the stochastic Bregman subgradient methods for solving constrained composite optimization problems and establish its convergence results. Finally, we conduct numerical experiments on training nonsmooth neural networks to evaluate the performance of our proposed stochastic Bregman subgradient methods. Our experimental results validate the practical benefits and effectiveness of our approaches in deep learning.

\vspace{2em}

\noindent\textbf{Funding} The research of Kim-Chuan Toh is supported by the Ministry of Education, Singapore, under its Academic Research Fund Tier 2 grant (MOE-T2EP20224-0029).

\vspace{2em}

\noindent\textbf{Data Availibility} The datasets analyzed during the current study are publicly available in the following resources:
\begin{itemize}
    \item The CIFAR-10 and CIFAR-100 datasets, introduced by \cite{krizhevsky2009learning}, can be accessed online at \url{https://www.cs.toronto.edu/~kriz/cifar.html}.
    \item The Penn Treebank dataset, described by \cite{marcus1993building}, is available at \url{https://catalog.ldc.upenn.edu/LDC99T42}.
\end{itemize}

\appendix  
\renewcommand{\thesection}{Appendix \Alph{section}:}
\renewcommand{\thelemma}{\Alph{section}.\arabic{lemma}}

\section{Proof of Theorem \ref{subsequential-thm}}
To establish under Assumption \ref{Assumption: DI} that, the discrete sequence generated by the general Bregman-type method \eqref{Eq:Breg_general_iterative} tracks a trajectory of the differential inclusion \eqref{Eq:Breg_DI}, we introduce a piecewise constant mapping $d(\cdot)$ defined by $d(s)=d_k$ for any $s\in[\lambda_\eta(k),\lambda_\eta(k+1))$. We also define a time-shifted solution $x_t(\cdot)$ to the following ordinary differential equation:
\[
\frac{d}{ds}\nabla\phi({x_t(s)})=-d(s)\;for\;all\;s\geq t,\;\mbox{with initial condition}\;\nabla\phi(x_t(t))=\nabla\phi(x(t)).
\]
The following lemma suggests that the interpolated process $x(t)$ defined by \eqref{Eq:interpolation} asymptotically approximates this time-shifted solution.
We remind the reader to note the difference between $x_t(\cdot)$ and $x^t(\cdot)$.
\begin{lemma}
\label{le:shift convergence}
Suppose Assumption \ref{Assumption: DI} holds, then for any $T>0$, it holds that
\begin{equation}
\lim_{t\rightarrow\infty}\sup_{s\in[t,t+T]}\|\nabla\phi(x(s))-\nabla\phi(x_t(s))\|=0.
\label{asymptotic-lemma}
\end{equation}
\end{lemma}
\begin{proof}
Fix an arbitrary $s\in[t,t+T]$, let $\tau_t=\Lambda_\eta(t)$, $\tau_s=\Lambda_\eta(s)$. By the definition of $x_t(\cdot)$, we have
\[
\begin{aligned}
&\nabla\phi(x_t(s))=\nabla\phi(x(t))-\int_{t}^{\lambda_\eta(\tau_t)}d(u)du-\int_{\lambda_\eta(\tau_t)}^{\lambda_\eta(\tau_s)}d(u)du-\int_{\lambda_\eta(\tau_s)}^s d(u)du\\
=&\nabla\phi(x(\lambda_\eta(\tau_t)))-\sum_{i=\tau_t}^{\tau_s-1}\eta_id_i+\left(\nabla\phi(x(t))-\int_t^{\lambda_\eta(\tau_t)}d(u)du-\nabla\phi(x(\lambda_\eta(\tau_t)))\right)-\int_{\lambda_\eta(\tau_s)}^s d(u)du\\
=&\nabla\phi(x({\lambda_\eta(\tau_s})))+\sum_{i=\tau_t}^{\tau_s-1}\eta_i\xi_i+\left(\nabla\phi(x(t))-\int_t^{\lambda_\eta(\tau_t)}d(u)du-\nabla\phi(x(\lambda_\eta(\tau_t)))\right)-\int_{\lambda_\eta(\tau_s)}^s d(u)du\\
=&\nabla\phi(x(s))+\sum_{i=\tau_t}^{\tau_s-1}\eta_i\xi_i+\left(\nabla\phi(x(t))-\int_t^{\lambda_\eta(\tau_t)}d(u)du-\nabla\phi(x(\lambda_\eta(\tau_t)))\right)\\
&+\left(\nabla\phi(x(\lambda_\eta(\tau_s)))-\int_{\lambda_\eta(\tau_s)}^sd(u)du-\nabla\phi(x(s))\right).
\end{aligned}
\]
Note that
\[
\begin{aligned}
&\norm{\nabla\phi(x(t))-\int_t^{\lambda_\eta(\tau_t)}d(u)du-\nabla\phi(x(\lambda_\eta(\tau_t)))}\\
\leq&{\|\nabla\phi(x(t))-\nabla\phi(x(\lambda_\eta(\tau_t)))\|}
+\int_t^{\lambda_\eta(\tau_t)}\|d(u)\|du
\\
 \leq&{\|\nabla\phi(x(\lambda_\eta(\tau_t+1)))-\nabla\phi(x(\lambda_\eta(\tau_t)))\|}
+\int_t^{\lambda_\eta(\tau_t)}\|d(u)\|du
\\
 \leq&\eta_{\tau_t}(\|\xi_{\tau_t}\|+2\|d_{\tau_t}\|),
 \end{aligned}
\]
and similarly 
$
\big\|\nabla\phi(x(\lambda_\eta(\tau_s)))-\int_{\lambda_\eta(\tau_s)}^sd(u)du-\nabla\phi(x(s))\big\|\leq\eta_{\tau_s}(\|\xi_{\tau_s}\|+2\|d_{\tau_s}\|).
$
By Assumption \ref{Assumption: DI}, we 
have that $\lim\sup\limits_{t\rightarrow\infty}\eta_{\tau_t}(\|\xi_{\tau_t}\|+2\|d_{\tau_t}\|)=0$, $\lim\sup\limits_{s\rightarrow\infty}\eta_{\tau_s}(\|\xi_{\tau_s}\|+2\|d_{\tau_s}\|)=0$, and $\lim\limits_{t\rightarrow\infty}\sup\limits_{s\in[t,t+T]}\sum_{i=\tau_t}^{\tau_s-1}\eta_i\xi_i=0$. Therefore, it holds that $\lim\limits_{t\rightarrow\infty}\sup\limits_{s\in[t,t+T]}\|\nabla\phi(x(s))-\nabla\phi(x_t(s))\|=0$, which completes the proof.
\end{proof}
\noindent{\textit{Proof of Theorem \ref{subsequential-thm}
.}} By the definition of $x_t(\cdot)$, it follows that $\nabla\phi(x_t(s))=\nabla\phi(x(t))-\int_t^s y(u)du$, for all $s\geq t$. By the boundedness of $y(s)$, Arzel$\grave{a}$-Ascoli's theorem \cite{rudin1953principles} ensures that $\{\nabla\phi(x_t(t+\cdot))\}_{t\in\R_+}$ is relatively compact in $\cC(\R_+,\R^n)$. For any subset $\{\tau_k\}\subset\R_+$, we consider the sequence $\{\nabla\phi(x^{\tau_k}(\cdot))\}_{k\in\mathbb{N}_+}$. There are two cases to consider. Case (i): the sequence $\{\tau_k\}$ has a cluster point $t$. Without loss of generality, assume that $\lim_{k\rightarrow\infty}\tau_k=t$. By the definition of $x(\cdot)$ {in \eqref{Eq:interpolation}} and uniform boundedness assumption, it holds that  $\nabla\phi(x(\cdot))$ is Lipschitz continuous. Thus, $\nabla\phi(x^{\tau_k}(\cdot))$ converges to $\nabla\phi(x^t(\cdot))$ in $\cC(\R_+,\R^n)$. Case (ii): $\lim_{k\rightarrow\infty}\tau_k=\infty$. Suppose that $\nabla\phi(x^{\tau_k}(\cdot))$ does not has any cluster point in $\cC(\R_+,\R^n)$. Since $\{\nabla\phi(x_{\tau_k}(\tau_k+\cdot))\}$ is relatively compact in $\cC(\R_+,\R^n)$, without loss of generality, we assume that $\lim_{k\rightarrow\infty}\nabla\phi(x_{\tau_k}(\tau_k+\cdot))=\bar y(\cdot)$. Then, for any compact set $C\subset\R_+$, it follows from Lemma \ref{le:shift convergence} that
\[
\begin{aligned}
&\lim_{k\rightarrow\infty,s\in C}\norm{\nabla\phi(x^{\tau_k}(s))-\bar y(s)}\\
\leq&\lim_{k\rightarrow\infty,s\in C}\norm{\nabla\phi(x^{\tau_k}(s))-\nabla\phi(x_{\tau_k}(\tau_k+s))}+\lim_{k\rightarrow\infty,s\in C}\norm{\nabla\phi(x_{\tau_k}(\tau_k+s))-\bar y(s)}=0,
\end{aligned}
\]
which contradicts that $\nabla\phi(x^{\tau_k}(\cdot))$ does not has any cluster point in $\cC(\R_+,\R^n)$. Thus, in both cases, $\{\nabla\phi(x^{\tau_k}(\cdot))\}$ is relatively compact in $\cC(\R_+,\R^n)$. Because $\{\tau_k\}$ is an arbitrary subset in $\R_+$, we have that $\{\nabla\phi(x^t(\cdot))\}$ is relatively compact in $\cC(\R_+,\R^n)$. 

Next, we aim to construct a trajectory of the differential inclusion. Define the shifts $d^t(\cdot)=d(t+\,\cdot)$. Consider $\{\tau_k\}$ satisfying $\tau_k\rightarrow\infty$, and fix $T>0$. Without loss of generality, we assume that $\nabla\phi(x^{\tau_k}(\cdot))$ converges to $\bar y(\cdot)$ in $\cC(\R_+,\R^n)$, otherwise, we choose its convergent subsequence. The set $\mathcal{Y}_T:=\{d^{\tau_k}(s),s\in[0,T]\}_{k\in\mathbb{N}}\subset L^2([0,T])$ is bounded. Therefore, it follows from the Banach-Alaoglu theorem \cite{rudin1991functional} that $\cY_T$ is weakly sequentially compact, i.e. there exists a subsequence $\{\tau_{k_j}\}$ and $\bar y(\cdot)\in L^2([0,T])$ such that $d^{\tau_{k_j}}(\cdot)\rightarrow\bar d(\cdot)$ weakly in $L^2([0,T])$. On the other hand, by Lemma \ref{le:shift convergence}, we have $\nabla\phi(x_{\tau_{k_j}}(\tau_{k_j}+\cdot))$ 
converges to $\bar y(\cdot)$ in $\cC(\R_+,\R^n)$. For any $\tau\in[0,T]$, by the definition of $x_t(\cdot)$, we have
\[
\nabla\phi(x_t(t+\tau))=\nabla\phi(x_t(t))-\int_0^\tau d^t(s)ds.
\]
Setting $t=\tau_{k_j}$ and taking the limit as $k_j\rightarrow\infty$, we deduce that
\[
\bar y(\tau)=\bar y(0)-\int_0^\tau\bar d(s)ds.
\]
Let $\bar x(\cdot)=\nabla\phi^*(\bar y(\cdot))$, since $T>0$ is arbitrary, we get \eqref{Eq:DI_int_form}. 

The remaining step is to verify that $\bar d(s)\in\mathcal{H}(\bar x(s))$ for almost all $s\geq0$. We again fix an arbitrary $T > 0$. Given that $\mathcal{Y}_T\subset L^2([0,T])$ is bounded, the Banach-Saks theorem \cite{rudin1991functional} implies that for ${\tau_k}$ (choosing a subsequence if necessary), $\frac{1}{N}\sum_{k=1}^Nd^{\tau_k}(s)$ strongly converges to $\bar d(s)$ in $L^2([0,T])$ . By the definition of $d(\cdot)$, we have $d^{\tau_k}(s)=d_{\Lambda_\eta(\tau_k+s)}$. Now for any $s\in[0,T]$, we have
\[
\begin{aligned}
    &\|\nabla\phi(x(\lambda_\eta(\Lambda_\eta(\tau_k+s))))-\nabla\phi(\bar x(s))\|\\
    \leq&\|\nabla\phi(x(\lambda_\eta(\Lambda_\eta(\tau_k+s))))-\nabla\phi(x(\tau_k+s))\|+\|\nabla\phi(x^{\tau_k}(s))-\nabla\phi(\bar x(s))\|\\
    \leq&\|\nabla\phi(x(\lambda_\eta(\Lambda_\eta(\tau_k+s)+1)))-\nabla\phi(x(\lambda_\eta(\Lambda_\eta(\tau_k+s))))\|+\|\nabla\phi(x^{\tau_k}(s))-\nabla\phi(\bar x(s))\|\\
    \leq&\eta_{\Lambda_\eta(\tau_k+s)}(\|\xi_{\Lambda_\eta(\tau_k+s)}\|+\|d_{\Lambda_\eta(\tau_k+s)}\|)+\|\nabla\phi(x^{\tau_k}(s))-\nabla\phi(\bar x(s))\|,
\end{aligned}
\]
which converges to zero as $k\rightarrow\infty$. By the continuity of $\nabla\phi^*$, we have that $x(\lambda_\eta(\Lambda_\eta(\tau_k+s)))$ converges to $\bar x(s)$ in $\cC(\R_+,\R^n)$. By Assumption \ref{Assumption: DI}, for almost any $s\in[0,T]$, we have
\[
\begin{aligned}
{\rm dist}(\bar d(s),\mathcal{H}(\bar x(s)))&\leq\bigg\|\frac{1}{N}\sum_{k=1}^Nd^{\tau_k}(s)-\bar d(s)\bigg\|+{\rm dist}\left(\frac{1}{N}\sum_{k=1}^Nd^{\tau_k}(s),\mathcal{H}(\bar x(s))\right)\\
=&\bigg\|\frac{1}{N}\sum_{k=1}^Nd^{\tau_k}(s)-\bar d(s)\bigg\|+
{\rm dist}\left(\frac{1}{N}\sum_{k=1}^Nd_{\Lambda_\eta(\lambda_\eta(\tau_k+s))},\mathcal{H}(\bar x(s))\right)\rightarrow0.
\end{aligned}
\]
Since $T$ is arbitrary and $\mathcal{H}(\bar x(s))$ is a closed set, we conclude that $\bar d(s)\in\mathcal{H}(\bar x(s))$ for almost all $s\geq0$. This completes the proof.\qed

\section{Proof of Theorem \ref{convergence-thm-func-val}}\label{append:thm-convergence}
\begin{lemma}
\label{le:liminf_limsup}
Suppose Assumption \ref{Assumption: DI} is satisfied. Then, it holds that
\begin{equation}
    \label{Eq:liminf_limsup}
    \liminf_{t\rightarrow\infty}\Psi(x(t))=\liminf_{k\rightarrow\infty}\Psi(x_k),\quad \limsup_{t\rightarrow\infty}\Psi(x(t))=\limsup_{k\rightarrow\infty}\Psi(x_k).
\end{equation}
\end{lemma}
\begin{proof}
For simplicity, we only prove the case for $\liminf$, the argument for $\limsup$ follows similarly. By Assumption \ref{Assumption: DI}, we have that 
\begin{equation}
\lim_{k\rightarrow\infty}\|\nabla\phi(x_{k+1})-\nabla\phi(x_k)\|=0.
\label{Eq:succ_gap_convergence}
\end{equation}
Let $\tau_i\rightarrow\infty$ be an arbitrary sequence with $x(\tau_i)\rightarrow x^*$. By the definition of $\lambda_\eta$ and $\Lambda_\eta$, it follows that
\[
\begin{aligned}
\|\nabla\phi(x_{\Lambda_\eta(\tau_i)})-\nabla\phi(x^*)\|\leq&\|\nabla\phi(x_{\Lambda_\eta(\tau_i)})-\nabla\phi(x(\tau_i))\|+\|\nabla\phi(x(\tau_i))-\nabla\phi(x^*)\|\\
\leq&\|\nabla\phi(x_{\Lambda_\eta(\tau_i)})-\nabla\phi(x_{\Lambda_\eta(\tau_i)+1})\|+\|\nabla\phi(x(\tau_i))-\nabla\phi(x^*)\|.
\end{aligned}
\]
The right-hand side converges to zero. By Remark \ref{rmk:assumption_DI}.1, we have that $\Psi\circ\nabla\phi^*$ is continuous, so $\lim_{i\rightarrow\infty}\Psi(x_{\Lambda_\eta(\tau_i)})=\lim_{i\rightarrow\infty}\Psi\circ\nabla\phi^*\circ\nabla\phi(x_{\Lambda_\eta(\tau_i)})=\Psi\circ\nabla\phi^*\circ\nabla\phi(x^*)=\Psi(x^*)$. By choosing $\tau_i\rightarrow\infty$ as the sequence realizing $\liminf_{t\rightarrow\infty}\Psi(x(t))$, and assuming without loss of generality that $x(\tau_i)\rightarrow x^*$, we get
\[   \liminf_{k\rightarrow\infty}\Psi(x_k)\leq
\lim_{{i\rightarrow\infty}}\Psi(x_{\Lambda_\eta(\tau_i)})=\Psi(x^*)
=\liminf_{t\rightarrow\infty}\Psi(x(t)).
\]
This completes the proof.
\end{proof}

The following lemma demonstrates that the function value converges along the interpolated process  defined in \eqref{Eq:interpolation}. The non-escape argument in the proof is adapted from those of \cite[Proposition 3.5]{davis2020stochastic} and \cite[Theorem 3.20]{duchi2018stochastic}, with particular attention paid to the dual map $\nabla\phi$ and its inverse $\nabla\phi^*$.
\begin{lemma}
    \label{prop:nonescape}
    Suppose Assumption \ref{Assumption: DI} and \ref{assumption_Sard_Lyapunov} hold, then function value $\Psi(x(t))$ converges {as $t\to\infty$.}
\end{lemma}
\begin{proof}
Assuming $\liminf_{t\rightarrow\infty}\Psi(x(t))=0$, we define the level set $\cL_r:=\{x\in\R^n:\;\Psi(x)\leq r\}$. Choose any $\epsilon>0$ such that $\epsilon\notin\Psi(\mathcal{H}^{-1}(0))$. The weak Morse-Sard condition in Assumption \ref{Assumption: DI} implies that $\epsilon$ can be chosen arbitrarily small, and Lemma \ref{le:liminf_limsup} implies that there are infinitely many $k$ such that $x_k\in\cL_\epsilon$. For any $x_k\in\cL_\epsilon$, by the continuity of $\Psi$, we have that there exists $\alpha>0$ such that ${\rm dist}(x_k,\R^n\setminus\cL_{2\epsilon})>\alpha$. By \eqref{Eq:succ_gap_convergence}, for sufficiently large $k$, we have that $\norm{x_{k+1}-x_k}<\alpha$. Therefore, for all large $k$, $x_k\in\cL_\epsilon$ implies that $x_{k+1}\in\cL_{2\epsilon}$.
Now, we define the last entrance and the first subsequent exit times,
\begin{equation}
\begin{aligned}
    &k_i=\max\{m\geq j_{i-1}:x_m\in\cL_\epsilon\}, \quad
    &j_i=\min\{m\geq k_i:x_m\in\R^n\setminus\cL_{2\epsilon}\}.
\end{aligned}
\label{Eq:upcross_def}
\end{equation}
We prove that such upcrossing occurs for finite times. Otherwise, if there exists $\{k_i\}$ such that $\lim_{i\rightarrow\infty}k_i=\infty$, then Theorem \ref{subsequential-thm} indicates that, up to a subsequence, $\nabla\phi(x^{\lambda_\eta(k_i)}(\cdot))$ converges to $\nabla\phi(\bar x(\cdot))$, where $\bar x(\cdot)$ is a trajectory of \eqref{Eq:Breg_DI}. By the definition of $k_i$, we have $\Psi(x_{k_i})\leq\epsilon$, $\Psi(x_{k_i+1})>\epsilon$. By \eqref{Eq:succ_gap_convergence} and the continuity of $\Psi\circ\nabla\phi^*$, we have that $\lim_{i\rightarrow\infty}\Psi(x_{k_i})=\lim_{i\rightarrow\infty}\Psi(x_{k_i+1})=\epsilon$. Recall that $x^{\lambda_\eta(k_i)}(0)=x_{k_i}$, therefore, $\Psi(\bar x(0))=\lim_{i\rightarrow\infty}\Psi(x_{k_i})=\epsilon$. Since $\bar x(0)$ is not in the stable set, there exists $T>0$, such that  
\[
\Psi(\bar x(T))<\sup_{s\in[0,T]}\Psi(\bar x(s))\leq\Psi(\bar x(0))=\epsilon.
\]
Then, there exists $\delta>0$, such that $\Psi(\bar x(T))\leq\epsilon-2\delta$. Moreover, for sufficiently large $i$, we have
\[
\sup_{s\in[0,T]}\Psi(x^{\lambda_\eta(k_i)}(s))\leq\sup_{s\in[0,T]}\Psi(\bar x(s))+\sup_{s\in[0,T]}|\Psi(x^{\lambda_\eta(k_i)}(s))-\Psi(\bar x(s))|\leq2\epsilon.
\]
The last inequality comes from the uniform convergence of $\{\nabla\phi(x^{\lambda_\eta(k_i)}(\cdot))\}$ in $\cC(\R_+,\R^n)$. This implies that for all large $i$, $\{x(\lambda_\eta(k_i)+s):s\in[0,T]\}\subset\cL_{2\epsilon}$. Thus it holds that $\lambda_\eta(j_i)>\lambda_\eta(k_i)+T$. Let $l_i=\max\{m:\lambda_\eta(k_i)\leq \lambda_\eta(m)\leq \lambda_\eta(k_i)+T\}$. Then $\|\nabla\phi(x_{l_i})-\nabla\phi(x^{\lambda_\eta(k_i)}(T))\|\leq\|\nabla\phi(x_{l_i})-\nabla\phi(x_{l_i+1})\|\rightarrow0$, and hence $\norm{\nabla\phi(x_{l_i})-\nabla\phi(\bar x(T))}\rightarrow0$ {as $i\to\infty$}. By the continuity of $\Psi\circ\nabla\phi^*$, we have that $\Psi(x_{l_i})\leq\epsilon-\delta$ for all large $i$. By the definition of $k_i$ and $j_i$, we have that $\lambda_\eta(j_i)<\lambda_\eta(l_i)\leq\lambda_\eta(k_i)+T$, which leads to a contradiction. Therefore, for all large $k$, $\Psi(x_k)\leq2\epsilon$. Since $\epsilon$ can be chosen arbitrarily small, it holds that $\lim_{k\rightarrow\infty}\Psi(x_k)=0$. This completes the proof.
\end{proof}


\section{Proof of Theorem \ref{The_convergence_Nondiminishing_RR}}\label{append:stability}
{
\begin{lemma}
\label{Le_controlled_noise_RR_appendix}
Suppose Assumption \ref{Assumption_Reshuffling} holds for the sequence of noises $\{\xi_{k}\}$ and stepsizes $\{\eta_k\}$. Moreover, the Lyapunov function $\Psi$ associated with $\cH$ is coercive. Then for any $\varepsilon > 0$ and $T>0$, there exists $\eta_{\varepsilon} > 0$ such that for any $\{\eta_k\}$ satisfying $\limsup_{k\to +\infty} \eta_k \leq \eta_{\varepsilon}$, almost surely, it holds that 
\begin{equation}
    \limsup_{s\to +\infty} \sup_{s \leq i\leq {\Lambda_\eta(\lambda_\eta(s) + T)} }  \norm{ \sum_{k = s}^i \eta_k \xi_{k}} \leq \varepsilon.  
\end{equation}
\end{lemma}
\begin{proof}
From Assumption \ref{Assumption_Reshuffling}, it holds for all $s \geq 0$ and any $i$ satisfying $s \leq i\leq {\Lambda_\eta(\lambda_\eta(s) + T)}$ that 
\begin{equation}
\begin{aligned}
&\norm{ \sum_{k = s}^i \eta_k \xi_{k}} \leq \norm{ \sum_{k = s}^{N\cdot \lceil \frac{s}{N} \rceil-1} \eta_k \xi_{k}} + \norm{ \sum_{k = N\cdot \lceil \frac{s}{N} \rceil}^{N\cdot \lfloor \frac{i}{N} \rfloor-1} \eta_k \xi_{k}} + \norm{ \sum_{k=N\cdot \lfloor \frac{i}{N} \rfloor}^i \eta_k \xi_{k}}\\
={}& \norm{ \sum_{k = s}^{N\cdot \lceil \frac{s}{N} \rceil-1} \eta_k \xi_{k}} + \norm{ \sum_{k=N\cdot \lfloor \frac{i}{N} \rfloor}^i \eta_k \xi_{k}}.
\end{aligned}
\end{equation}
Let $M_{\xi}$ be the uniform bound of $\{\xi_k\}$. Therefore, for any any $\varepsilon > 0$, choosing $\eta_{\varepsilon} = \frac{\varepsilon}{2NM_{\xi}}$ guarantees that 
\begin{equation*}
\limsup_{s\to +\infty} \sup_{s \leq i\leq \Lambda_\eta(\lambda_\eta(s) + T) }  \norm{ \sum_{k = s}^i \eta_k \xi_{k}} \leq 2NM_{\xi}\limsup_{s\to +\infty, ~s \leq i\leq \Lambda_\eta(\lambda_\eta(s) + T) }  \eta_{i} \leq \varepsilon. 
\end{equation*}
This completes the proof.
\end{proof}

\noindent{\textit{Proof of Theorem \ref{The_convergence_Nondiminishing_RR}
.}} For the differential inclusion \eqref{Eq:Breg_DI} and the update scheme \eqref{Eq:Breg_general_iterative}, define $z_k=\nabla\phi(x_k)$ and $z(t)=\nabla\phi(x(t))$. Then, we have 
\begin{equation}
\dot z\in-\cH(\nabla\phi^*(z))\text{ and }z_{k+1}=z_k-\eta_k(d_k+\xi_k). 
\end{equation}
Note that $\cH\circ\nabla\phi^*$ has a closed graph. By assumption, $\phi$ is $\mu$-strongly convex and $L$-smooth, with some $\mu>0$. Set $\tilde\delta_k:=\max\{\delta_k,L\delta_k\}$. By $L$-smoothness of $\phi$, We have 
\[
(\cH\circ\nabla\phi^*)^{\tilde\delta_k}(z_k)=\cH(\nabla\phi^*(z_k+\tilde\delta_k\mathbb{B}^n))+\tilde\delta_k\mathbb{B}^n\supset\cH\left(\nabla\phi^*(z_k)+\frac{\tilde\delta_k}{L}\mathbb{B}^n)\right)+\tilde\delta_k\mathbb{B}^n\supset \cH(\nabla\phi^*(z_k)+\delta_k\mathbb{B}^n)+\delta_k\mathbb{B}^n.
\]
Note that $d_k\in\cH^{\delta_k}(x_k)=\cH(\nabla\phi^*(z_k)+\delta_k\mathbb{B}^n)+\delta_k\mathbb{B}^n$, it follows that $d_k\in(\cH\circ\nabla\phi^*)^{\tilde\delta_k}(z_k)$. Lemma \ref{Le_controlled_noise_RR} and \cite[Theorem 3.5]{xiao2023convergence} illustrate that for any $\tilde\varepsilon > 0$, there exists ${\eta}_{\max}, T > 0$ such that whenever $\lim\sup_{k\geq 0} \eta_k \leq{\eta}_{\max}$ and $\{\xi_k\}$ is $(\tilde\varepsilon, T, \{\eta_k\})$-controlled, then $\{z_k\}$ is uniformly bounded and
\[
\limsup_{k\to +\infty}\;\mathrm{dist}\left( z_k, \{z \in \R^n: 0\in\cH(\nabla\phi^*(z))\}  \right) \leq {\tilde\varepsilon}. 
\]
Because $\phi$ is $\mu$ strongly convex, we have 
\[
\limsup_{k\to +\infty}\;\mathrm{dist}\left( x_k, \{x \in \R^n: 0\in\cH(x)\}  \right) \leq \limsup_{k\to +\infty}\;\frac{1}{\mu}\mathrm{dist}\left( z_k, \{x \in \R^n: 0\in\cH(\nabla\phi^*(z))\}  \right) \leq \frac{\tilde\varepsilon}{\mu}. 
\]
Since $\tilde\varepsilon$ is arbitrary, this completes the proof.
}

\bibliographystyle{spmpsci}

\end{document}